%% file: MultilevelDecompositions.tex
\documentclass[11pt]{amsart}

\usepackage[utf8]{inputenc}
\usepackage[T1]{fontenc}

\usepackage{amsmath,amssymb,amsfonts,textcomp,amsthm,xifthen,graphicx,color,pgfplots}

\usepackage{enumerate}
\usepackage{enumitem}

\usepackage{fullpage}

\usepackage[utf8]{inputenc}

\usepackage[pdftex,
            pdfauthor={Thomas F\"uhrer},
            pdftitle={Multilevel decompositions and norms for negative order Sobolev spaces},
            ]{hyperref}

\usepackage{pgfplots}
\usepgfplotslibrary{external}
\usepgfplotslibrary{colorbrewer}
\input{header}

\begin{document}

\title{Multilevel decompositions and norms for negative order Sobolev spaces}
\date{\today}

\author{Thomas F\"uhrer}
\address{Facultad de Matem\'{a}ticas, Pontificia Universidad Cat\'{o}lica de Chile, Santiago, Chile}
\email{tofuhrer@mat.uc.cl}

\thanks{{\bf Acknowledgment.} 
This work was supported by ANID through FONDECYT projects 11170050 and 1210391.}

\begin{abstract}
  We consider multilevel decompositions of piecewise constants on simplicial meshes that are stable in $H^{-s}$ for $s\in (0,1)$. 
  Proofs are given in the case of uniformly and locally refined meshes. 
  Our findings can be applied to define local multilevel diagonal preconditioners that lead to bounded condition numbers (independent of the mesh-sizes and levels) and have optimal computational complexity. 
  Furthermore, we discuss multilevel norms based on local (quasi-)projection operators that allow the efficient evaluation of negative order Sobolev norms. 
  Numerical examples and a discussion on several extensions and applications conclude this article. 
\end{abstract}
\maketitle

\section{Introduction}
This work deals with the analysis of multilevel decompositions and multilevel norms of piecewise constant functions for the Sobolev spaces $H^{-s}(\Omega)$ resp. $\widetilde H^{-s}(\Omega)$ with $s\in(0,1)$. 

Stability results for subspace decompositions are needed in the analysis of, e.g., additive and multiplicative Schwarz preconditioners, see~\cite{oswald94,ToselliWidlund} for an overview. 
An important use case is the definition of additive Schwarz preconditioners for weakly-singular integral equations~\cite{TranS_96_ASM,HeuerST_98_MAS,Heuer_01_ApS}. The two dimensional case (1 dimensional boundary) follows from corresponding results in Sobolev spaces of opposite (and therefore positive) order whereas the higher dimensional case needs different techniques. 
Another application is given in fictitious domain methods~\cite{BerroneBonitoStevenson19}.
In our recent work~\cite{ABEMsolve} we have defined a multilevel diagonal preconditioner for the weakly-singular integral operator which is optimal on locally refined meshes and closed boundaries for three-dimensional problems.
The proofs are based on the abstract framework from~\cite{oswald99}.
Although verified numerically, optimality for open boundaries is not shown in~\cite{ABEMsolve}. 
Moreover, it is not clear if the techniques given in~\cite{ABEMsolve} extend to the general case with $s\in(0,1)$. 
In the recent work~\cite{BaerlandKuchtaMardal19}, additive multigrid methods are analyzed for problems involving the fractional Laplacian leading to level dependent condition number estimates.

A different approach is the framework of operator preconditioning, see~\cite{Hiptmair06} for an overview. 
One advantage is that the history of meshes is usually not needed. 
One drawback is the use of dual meshes and discretized operators of opposite order which often is computationally expensive.
The latter issues have been tackled in a series of recent articles~\cite{StevensonVanVenetie2020mathComp,StevensonVanVenetie2020posOrder,StevensonVanVenetie2020negOrder}.

Multilevel norms for negative order Sobolev spaces and piecewise constant functions have been analyzed in~\cite{Oswald98} but do not lead to level independent estimates for $s\geq 1/2$. 
Multilevel norms for piecewise affine and globally continuous functions are found in, e.g.,~\cite{BornemannYserentant,BramblePasciakVassilevski00,Oswald98}, see also the recent article~\cite{FaustmannMelenk21}.
Multigrid methods are analyzed in, e.g.,~\cite{WuChen06,HiptmairWuZheng2012,CNX12,CHX13}.
Other works that use a matrix-based approach to treat the evaluation of fractional Sobolev norms include~\cite{Burstedde07,ArioliLoghin09} but rely on the evaluation of fractional powers of non-trivial matrices or the use of wavelet bases. 
Wavelet techniques are used for boundary element methods, see, e.g.,~\cite{SchmidlinSchwab02,HarbrechtSchneider04,DHS07}.
More details on the theory of (pre-)wavelet methods are found in, e.g.,~\cite{Stevenson96,Stevenson98,CTU99}.

%
\subsection{Some known results on multilevel norms}\label{sec:knownresults:multilevel}
We recall some results on multilevel norms from~\cite{Oswald98} adopted to the notation used in the present work.
Let $(\TT_\ell)_{\ell\in\N_0}$ denote a sequence of uniformly refined simplicial meshes with mesh sizes $(h_\ell)_{\ell\in\N_0}$.
For the range $s\in(0,1/2)$ one gets from~\cite[Eq.(3)]{Oswald98} that
\begin{align}\label{eq:mlnL2}
  \norm{\phi}{-s}^2 \simeq \norm{\phi}{-s,\sim}^2 \simeq \sum_{\ell=0}^\infty h_\ell^{2s}\norm{(\projLtwo_\ell-\projLtwo_{\ell-1})\phi}{}^2,
\end{align}
where $\projLtwo_\ell$ denotes the $L^2(\Omega)$ orthogonal projection on the space of piecewise constants $\PP^0(\TT_\ell)$. 
The multilevel norm can be efficiently evaluated if $\phi\in\PP^0(\TT_L)$. 

However, the latter equivalence does not include one of the arguably most important cases, $s=1/2$. 
It is shown (\cite[Theorem~2]{Oswald98}) that
\begin{align*}
  \norm{\phi}{-1/2}^2 \lesssim \sum_{\ell=0}^L h_\ell \norm{(\projLtwo_\ell-\projLtwo_{\ell-1})\phi}{}^2
  \lesssim (L+1)^2 \norm{\phi}{-1/2}^2 \quad\text{for all }\phi \in \PP^0(\TT_L)
\end{align*}
and the factor $(L+1)^2$ can not be improved in general, thus, yielding suboptimal results.

\subsection{Novel contributions}
Rather than using duality arguments to transfer results for positive order Sobolev spaces to negative order spaces (see e.g.~\cite[Section~2]{Oswald98}), 
we exploit the deep connection between interpolation and approximation spaces where we view $H^{-s}(\Omega)$ as an intermediate space between $H^{-1}(\Omega)$ and $L^2(\Omega)$. 
At first glance this seems to further complicate the problem due to the necessity of handling the $H^{-1}(\Omega)$ norm.
However, some ideas resp. results from our recent work~\cite{FuehrerHeuerQuasiDiagonal19} and the work at hand show how to define local decompositions in the $H^{-1}(\Omega)$ norm by using Haar-type functions that can be written as the divergence of Raviart--Thomas functions. 
Furthermore, we introduce and analyze locally defined projection operators in $H^{-1}(\Omega)$ onto the space of piecewise constant functions. 
We note that for problems of positive order the Sobolev space $H^s(\Omega)$ is an intermediate space between $L^2(\Omega)$ and $H^t(\Omega)$ for some $t\in (1,3/2)$. Together with stability of the canonical basis of piecewise affine and globally continuous functions in $L^2(\Omega)$ this yields optimality of the BPX preconditioner, see, e.g.,~\cite{Bornemann94,BornemannYserentant} and references therein for a detailed analysis.

Let us summarize two of our main results that can be found in Section~\ref{sec:main} and are valid for uniform as well as adaptive meshes:
Let $(\TT_\ell)_{\ell\in\N_0}$ denote a sequence of meshes with mesh-size functions $(h_\ell)_{\ell\in\N_0}$ and facets $(\EE_\ell)_{\ell\in\N_0}$. 
For $\ell\in\N_0$, $E\in\EE_\ell$ we define spaces $\XX_{\ell,E} = \linhull\{\psi_{\ell,E}\}\in\PP^0(\TT_\ell)$ where $\psi_{\ell,E}$ is supported on at most two elements of $\TT_\ell$ (which share the facet $E$) and $\int_\Omega \psi_{\ell,E} \,\mathrm{d}x = 0$ if $E\not\subset \partial\Omega$ (a precise definition is found in Section~\ref{sec:discretespaces}) and use subsets $\widetilde\EE_\ell\subset\EE_\ell$ that satisfy $\#\widetilde\EE_\ell=\OO(\#\TT_\ell\setminus\TT_{\ell-1})$. 
\begin{itemize}
  \item \textbf{Multilevel decomposition (Theorem~\ref{thm:main}):}  
    \begin{align*}
      \norm{\phi}{-s}^2 \simeq \inf\set{\sum_{\ell=0}^L \sum_{E\in\widetilde\EE_\ell} \norm{\phi_{\ell,E}}{-s}^2}{\phi_{\ell,E}\in \XX_{\ell,E} \text{ such that } 
      \phi = \sum_{\ell=0}^L \sum_{E\in\widetilde\EE_\ell} \phi_{\ell,E}}
    \end{align*}
    for all $\phi\in\PP^0(\TT_L)$.
  \item \textbf{Multilevel norms (Theorem~\ref{thm:multilevelnorm}):} There exist locally defined operators $\projHone_\ell'$ such that
    \begin{align*}
      \norm{\phi}{-s}^2 \simeq \sum_{\ell=0}^L \norm{h_\ell^s(\projHone_\ell'-\projHone_{\ell-1}')\phi}{}^2 \quad\text{for all }\phi \in \PP^0(\TT_L).
    \end{align*}
\end{itemize}
In particular, the constants are independent of the levels or the mesh-sizes. We note that the decomposition into one-dimensional subspaces implies that the associated preconditioner is of a multilevel diagonal scaling type. 

\subsection{Outline}
In Section~\ref{sec:prel} we introduce notation and some basic results. In particular, Section~\ref{sec:projHmOneSZ} and Section~\ref{sec:projHmOneSZtilde} deal with the definition and analysis of local projection operators in negative order Sobolev spaces. 
In Section~\ref{sec:main} we present our main results and their proofs are given in Section~\ref{sec:proof}.
Numerical experiments are presented in Section~\ref{sec:experiments}.
The final Section~\ref{sec:conclusion} concludes this article with a discussion on applications and extensions including the case of higher-order polynomial spaces. 

\subsection{Notation}
Throughout this work we write $a\lesssim b$ resp. $a\gtrsim b$ if there exists a constant $C>0$ such that $a\leq C b$ resp. $a\geq C b$. 
If both directions hold we write $a\simeq b$. 
In the main results dependencies on constants will be specified.

\section{Preliminaries}\label{sec:prel}

\subsection{Sobolev spaces} \label{sec:sobolev}
For a bounded Lipschitz domain $\omega\subset \R^d$ ($d\geq 2$) let 
\begin{align*}
  H^1(\omega) = \set{v\in L^2(\omega)}{\nabla u\in L^2(\omega)}
\end{align*}
with norm $\norm{\cdot}{H^1(\omega)} = \big(\norm{\cdot}{\omega}^2 + \norm{\nabla(\cdot)}{\omega}^2\big)^{1/2}$.
Here, $\norm{\cdot}{\omega}$ denotes the $L^2(\omega)$ or $L^2(\omega)^d$ norm which is induced by the scalar product $\ip{\cdot}{\cdot}_\omega$. 
Let $H_0^1(\omega)$ denote the closed subspace of $H^1(\omega)$ with vanishing traces and recall that $\norm{\nabla(\cdot)}\omega$ defines an equivalent norm on $H_0^1(\omega)$.
The dual space $H^{-1}(\omega):=(H_0^1(\omega))'$ is equipped with the dual norm
\begin{align*}
  \norm{\phi}{-1,\omega} := \sup_{0\neq v\in H_0^1(\omega)} \frac{\dual{\phi}{v}_\omega}{\norm{\nabla v}{\omega}},
\end{align*}
where the duality $\dual\cdot\cdot_\omega$ is understood with respect to the extended $L^2(\omega)$ inner product. 
Analogously, $\widetilde H^{-1}(\omega) := (H^1(\omega))'$ with norm
\begin{align*}
  \norm{\phi}{-1,\sim,\omega} := \sup_{0\neq v\in H^1(\omega)} \frac{\dual{\phi}{v}_\omega}{\big(\norm{\nabla v}{\omega}^2+\norm{v}{\omega}^2\big)^{1/2}}.
\end{align*}

Throughout this work we consider a connected Lipschitz domain $\emptyset\neq\Omega\subset\R^d$ with boundary $\Gamma:=\partial\Omega$ 
and skip indices in the notation of norms, e.g., we write $\norm{\cdot}{}$ instead of $\norm{\cdot}{\Omega}$, $\norm{\cdot}{-1}$ instead of $\norm{\cdot}{-1,\Omega}$.

For $s\in(0,1)$ we define the intermediate spaces $H^s(\Omega)$ resp. $\widetilde H^{s}(\Omega)$ by real interpolation ($K$-method, see, e.g.~\cite[Section~4]{CWHM15}), i.e., 
\begin{align*}
  H^s(\Omega) &:= \interp{L^2(\Omega)}{H^1(\Omega)}{s,2} \quad\text{resp.}\\
  \widetilde H^s(\Omega) &:= \interp{L^2(\Omega)}{H_0^1(\Omega)}{s,2}.
\end{align*}
We recall that the dual spaces $\widetilde H^{-s}(\Omega) := (H^s(\Omega))'$ resp. $H^{-s}(\Omega) := (\widetilde H^s(\Omega))'$ can be written as interpolation spaces as well, i.e.,
\begin{align*}
  \widetilde H^{-(1-\theta)}(\Omega) &= \interp{\widetilde H^{-1}(\Omega)}{L^2(\Omega)}{\theta,2}, \\
  H^{-(1-\theta)}(\Omega) &= \interp{H^{-1}(\Omega)}{L^2(\Omega)}{\theta,2}
\end{align*}
for all $\theta \in (0,1)$, see, e.g.,~\cite[Chapter~1, Theorem~6.2]{LionsMagenesI}.
An extensive overview on interpolation between Hilbert spaces and, in particular, Sobolev spaces is found in~\cite{CWHM15}.

The space of $L^2(\omega)^d$ fields with divergence in $L^2(\omega)$ is denoted by $\Hdivset\omega$.
Also note that $\div\colon L^2(\Omega)^d\to H^{-1}(\Omega)$ is a bounded operator. 

\subsection{Meshes and refinement}\label{sec:mesh}
Let $\TT$ denote a regular mesh of $\Omega$ consisting of open simplices, i.e., $\overline\Omega = \bigcup_{T\in\TT}\overline{T}$.
With $h_T := \diam(T)$ we define the mesh-size function $h_\TT$ by $h_\TT|_T := h_T$.
The set of all $d-1$ facets of an element $T\in\TT$ is denoted by $\EE(T)$. For our studies we also use the sets
\begin{align*}
  \EE_\TT := \bigcup_{T\in\TT} \EE(T), \quad \EE_\TT^\Gamma := \set{E\in\EE_\TT}{E\subset \Gamma}, 
  \quad \EE_\TT^\Omega := \EE_\TT\setminus\EE_\TT^\Gamma.
\end{align*}
The set of vertices of an element $T\in\TT$ is denoted with $\NN(T)$, and
\begin{align*}
  \NN_\TT := \bigcup_{T\in\TT} \NN(T), \quad \NN_\TT^\Gamma := \set{z\in\NN_\TT}{z\in\Gamma}, 
  \quad \NN_\TT^\Omega := \NN_\TT\setminus\NN_\TT^\Gamma.
\end{align*}
In this work we consider sequences of meshes $(\TT_\ell)_{\ell\in\N_0}$ where we assume that $\TT_{\ell+1}$ is generated from $\TT_\ell$ by refining certain (or all) elements.
A common type of mesh refinement is, e.g., \emph{newest vertex bisection}, see, e.g.,~\cite{Stevenson08} and references therein.
The generation $\gen(T)$ of an element $T\in \bigcup_{\ell\in\N_0} \TT_\ell$ denotes the number of iterated refinements (bisections) to obtain $T\in\TT_\ell$ from a father element $T'\in\TT_0$.
We assume, given a suitable initial mesh $\TT_0$, that the sequence $\TT_0,\TT_1,\dots$ generated by the mesh refinement strategy satisfies:
\begin{enumerate}[label=(A\arabic*)]
  \item\label{ass:mesh:reg} Shape regularity: There exists a constant $C_\mathrm{reg}>0$ such that
    \begin{align*}
      \sup_{\ell\in\N_0} \sup_{T\in\TT_\ell} \frac{\diam(T)^d}{|T|} \leq C_\mathrm{reg}.
    \end{align*}
  \item\label{ass:mesh:ref} There exists $q_\mathrm{ref}\in(0,1)$ and $C_\mathrm{ref}>0$ such that
    \begin{align*}
      C_\mathrm{ref}^{-1} h_T \leq q_\mathrm{ref}^{\gen(T)} \leq C_\mathrm{ref} h_T \quad\text{for all } T\in\bigcup_{\ell\in\N_0} \TT_\ell.
    \end{align*}
  \item\label{ass:mesh:gen} There exists a constant $k_\mathrm{ref}\in\N$ such that 
    for all $\ell\in\N_0$ and all $T\in\TT_{\ell+1}\setminus\TT_\ell$ with unique father element $T_F\in\TT_\ell$,
    \begin{align*}
      1\leq |\gen(T)-\gen(T_F)| \leq k_\mathrm{ref}.
    \end{align*}
\end{enumerate}
These assumptions are satisfied for, e.g., the newest vertex bisection, see~\cite{Stevenson08} and~\cite{GallistlSchedensackStevenson14}.

We say that $(\TT_m)_{m\in\N_0}$ is a sequence of uniform meshes if (besides the assumptions from above)
\begin{itemize}
  \item $\TT_{m+1}\setminus\TT_{m} = \TT_{m+1}$ and $\gen(T)=\gen(T')$ for all $T,T'\in \TT_m$ and $m\in\N_0$.
\end{itemize}
Moreover, we assume that for a sequence of meshes $(\TT_\ell)_{\ell\in\N_0}$ there exists a sequence of uniform meshes $(\widehat\TT_m)_{m\in\N_0}$ with $\widehat\TT_0=\TT_0$ satisfying~\ref{ass:mesh:reg}--\ref{ass:mesh:gen} with the same constants.

From the assumptions given in this section we can interpret the mesh-size functions of uniform meshes as constants.
\begin{lemma}\label{lem:meshgenunif}
  Let $(\TT_\ell)_{\ell\in\N_0}$ denote a sequence of uniform meshes, then
  \begin{align*}
    \gen(T) \simeq \ell \quad\text{and}\quad
    h_T\simeq q_\mathrm{ref}^\ell \quad\text{for all } T\in\TT_\ell, \,\ell\in\N_0.
  \end{align*}
  The constants involved only depend on the constants $C_\mathrm{ref}$, $k_\mathrm{ref}$ from ~\ref{ass:mesh:ref}--\ref{ass:mesh:gen}.
\end{lemma}
\begin{proof}
  Follows from Assumption~\ref{ass:mesh:gen} and Assumption~\ref{ass:mesh:ref}.
\end{proof}


Let $\TT$ be a regular mesh. Element patches are given by
\begin{align*}
  \patch_\TT(S) &:= \set{T\in\TT}{\overline{T}\cap \overline S \neq \emptyset} \quad\text{for some } S\subseteq \overline\Omega, \\
  \patch_\TT(z) &:= \patch_\TT(\{z\}) \quad\text{for some } z\in\overline\Omega.
\end{align*}
The corresponding domains are denoted with $\Omega_\TT(S)$ and $\Omega_\TT(z)$. 
Higher-order patches are denoted with an additional superscript, e.g., $\patch_\TT^{(2)}(S) = \patch_\TT(\Omega_\TT(S))$.

\subsection{Discrete spaces and projections}\label{sec:discretespaces}
For $T\in\TT$ we denote with $\PP^p(T)$ the space of polynomials of degree $\leq p\in\N_0$ and set
\begin{align*}
  \PP^p(\TT) := \set{v\in L^2(\Omega)}{v|_T \in \PP^p(T) \text{ for all } T\in\TT}.
\end{align*}
Furthermore,
\begin{align*}
  \cS^1(\TT) := \PP^1(\TT)\cap H^1(\Omega) \quad\text{and}\quad
  \cS_0^1(\TT) := \PP^1(\TT)\cap H_0^1(\Omega).
\end{align*}
The space $\cS^1(\TT)$ is equipped with the common basis $\set{\eta_{\TT,z}}{z\in\NN_\TT}$ where $\eta_{\TT,z}(z') = \delta_{z,z'}$ for all $z,z'\in\NN_\TT$.
Here, $\delta_{z,z'}$ denotes the Kronecker-$\delta$ symbol.

We make use of the $L^2(\Omega)$ orthogonal projection $\Pi^p_\TT\colon L^2(\Omega)\to \PP^p(\TT)$ and the $H^{-1}(\Omega)$ orthogonal projection $\projHmOne_\TT \colon H^{-1}(\Omega)\to \PP^0(\TT)$ resp. the $\widetilde H^{-1}(\Omega)$ orthogonal projection $\projTildeHmOne_\TT\colon \widetilde H^{-1}(\Omega)\to \PP^0(\TT)$.
It is well-known that these operators satisfy the approximation properties
\begin{align*}
  \norm{(1-\Pi^p_\TT)\phi}{-1} + \norm{(1-\Pi^p_\TT)\phi}{-1,\sim} 
  + \norm{(1-\projHmOne_\TT)\phi}{-1} + \norm{(1-\projTildeHmOne_\TT)\phi}{-1,\sim} &\lesssim \norm{h_\TT \phi}{}
\end{align*}
for all $\phi\in L^2(\Omega)$, which follow from duality and Poincar\'e inequalities.

Recall the inverse estimates
\begin{align}\label{eq:invest}
  \norm{h_\TT^s \phi}{} &\lesssim \norm{\phi}{-s} \lesssim \norm{\phi}{-s,\sim}
\end{align}
for all $\phi\in\PP^p(\TT)$ and $s\in[0,1]$, see, e.g.,~\cite[Theorem~3.6]{ghs05}, as well as the local variant
\begin{align*}
  h_T \norm{\phi}{T} &\lesssim \norm{\phi}{-1,T} \quad\text{for all } \phi\in \PP^p(T) \text{ and }T\in\TT,
\end{align*}
which follows from a scaling argument.
For a uniform mesh (recall that $h_\TT$ is constant) the interpolation inequality~\cite[Chapter~1, Proposition~2.8]{LionsMagenesI} and the inverse estimate~\eqref{eq:invest} show
\begin{align*}
  \norm{\phi}{-s} \lesssim \norm{\phi}{-1}^{s}\norm{\phi}{}^{1-s} \lesssim h_\TT^{-1+s}\norm{\phi}{-1}
  \quad\text{and}\quad
  \norm{\phi}{-s,\sim} \lesssim \norm{\phi}{-1,\sim}^{s}\norm{\phi}{}^{1-s} \lesssim h_\TT^{-1+s}\norm{\phi}{-1,\sim}
\end{align*}
for all $\phi\in \PP^p(\TT)$.

For the decompositions we use Haar-type functions which can be written as the divergence of Raviart--Thomas functions. Let $\RT^p(\TT)$ denote the space of Raviart--Thomas functions of order $p\in\N_0$.
For each $E\in\EE_\TT^\Omega$ there exist unique elements $T_E^\pm\in\TT$ with $E = \intr(\overline T_E^+\cap \overline T_E^-)$.
For each $E\in\EE_\TT^\Gamma$ there exists a unique element $T_E^+\in\TT$ with $E = \intr(\overline T_E^+ \cap \Gamma)$ and we set $T_E^-:=\emptyset$. 
Let $\chi_T$ denote the characteristic function of an element $T\in\TT$ and let $\set{\ppsi_{\TT,E}}{E\in\TT}$ be the canonical local basis of $\RT^0(\TT)$ with
$\supp \ppsi_{\TT,E} \subseteq \overline T^+ \cup \overline T^-$, vanishing normal trace on $\EE\setminus\{E\}$ and
\begin{align*}
  \psi_{\TT,E} := \div\ppsi_{\TT,E} = \frac{|E|}{|T_E^+|}\chi_{T_E^+} - \frac{|E|}{|T_E^-|}\chi_{T_E^-}.
\end{align*}
Here, $|E|$ denotes the surface measure of the facet $E\in\EE_\TT$ and we set $1/|T_E^-|=0$ for $E\in\EE_\TT^\Gamma$.
From~\cite[Lemma~1]{FuehrerHeuerQuasiDiagonal19} we recall some scaling properties, where $h_E:=\diam(E)$.
\begin{lemma}\label{lem:scaling}
  The equivalences
  \begin{alignat*}{2}
    \norm{\psi_{\TT,E}}{-1} &\leq \norm{\ppsi_{\TT,E}}{} \simeq h_E \norm{\psi_{\TT,E}}{} \lesssim \norm{\psi_{\TT,E}}{-1}
    &\quad&\text{for all } E\in\EE_\TT, \\
    \norm{\psi_{\TT,E}}{-1,\sim} &\leq \norm{\ppsi_{\TT,E}}{} \simeq h_E \norm{\psi_{\TT,E}}{} \lesssim \norm{\psi_{\TT,E}}{-1,\sim}
    &\quad&\text{for all } E\in\EE_\TT^\Omega
  \end{alignat*}
  hold and the involved constants only depend on the shape regularity of $\TT$ and the dimension $d$.
\end{lemma}
Let $s\in(0,1)$. Lemma~\ref{lem:scaling} together with interpolation and inverse estimates from above show that
\begin{alignat*}{2}
  \norm{\psi_{\TT,E}}{-s} &\simeq h_E^s \norm{\psi_{\TT,E}}{} &\quad&\text{for all } E\in\EE_\TT, \\
  \norm{\psi_{\TT,E}}{-s,\sim} &\simeq h_E^s \norm{\psi_{\TT,E}}{} &\quad&\text{for all } E\in\EE_\TT^\Omega.
\end{alignat*}

In the recent work~\cite{egsv2019} a local projection operator onto the Raviart--Thomas space has been defined that does not rely on regularity assumptions (the canonical Raviart--Thomas projection requires that $\ssigma\in H^t(\Omega)^d\cap\Hdivset\Omega$ with some $t>1/2$ so that the normal trace $\ssigma\cdot\normal|_E$ is well-defined in $L^2(E)$).
\begin{lemma}[{\cite[Theorem~3.2]{egsv2019}}]\label{lem:Hdivproj}
  Let $\emptyset\neq \widetilde\patch \subseteq \TT$ with $\widetilde\Patch = \intr(\bigcup_{T\in\widetilde\patch}\overline{T})$ a connected domain be given and set
  \begin{align*}
    \HdivsetZero{\widetilde\Patch} := \set{\ttau\in \Hdivset{\widetilde\Patch}}{\ttau\cdot\normal = 0 \text{ on }\partial\widetilde\Patch\setminus\Gamma}.
  \end{align*}
  There exists an operator $\projHdiv^p_{\widetilde\patch}\colon \HdivsetZero{\widetilde\Patch} \to \RT^p(\widetilde\patch)\cap \HdivsetZero{\widetilde\Patch}$ which satisfies
  \begin{alignat}{2}
    \projHdiv^p_{\widetilde\patch} \ssigma &= \ssigma &\quad&\text{for all } \ssigma \in \HdivsetZero{\widetilde\Patch}\cap \RT^p(\widetilde\patch), \\
    \div\projHdiv^p_{\widetilde\patch} \ssigma &= \Pi^0_{\widetilde\patch} \div\ssigma &\quad& \text{for all }\ssigma \in \HdivsetZero{\widetilde\Patch}.
  \end{alignat}
  Moreover,
  \begin{align}
    \norm{\projHdiv^p_{\widetilde\patch} \ssigma}T^2 &\lesssim \norm{\ssigma}{\Patch_{\widetilde\omega}(T)}^2 + 
    \norm{h_\TT(1-\Pi^p_{\widetilde\patch})\div\ssigma}{\Patch_{\widetilde\omega}(T)}^2
  \end{align}
  for all $T\in\widetilde\patch$ and $\ssigma\in\HdivsetZero{\widetilde\Patch}$.
  The involved constant only depends on the space dimension $d$, $p\in\N_0$ and the shape regularity of $\TT$. 
\end{lemma}

Finally, for a sequence $(\TT_\ell)_{\ell\in\N_0}$ we use the indices $\ell$ instead of $\TT_\ell$ in the notation of the corresponding operators, patches etc., e.g., $\projLtwo_\ell$ instead of $\projLtwo_{\TT_\ell}$ and $\patch_\ell(z)$ instead of $\patch_{\TT_\ell}(z)$.
Furthermore, we define operators with negative indices to be trivial, e.g., $\projLtwo_k := 0$ for $k<0$.

\subsection{A local projection operator in $H^{-1}(\Omega)$}\label{sec:projHmOneSZ}
In this section we define a local projection operator $\projHmOneSZ_\TT\colon H^{-1}(\Omega)\to \PP^0(\TT)$ which plays a crucial role in the stability analysis that follows. 
We give a brief overview of the basic idea: First, we consider a Fortin operator based on the quasi-interpolation operator~\cite{SZ_90}. Then, we study its adjoint operator and, finally, we use the canonical $L^2(\Omega)$ projection (on piecewise constants) to define $\projHmOneSZ_\TT$. 

We consider the following variant: For each $z\in\NN_\TT$ let $\emptyset\neq \gamma_z\subseteq\patch_\TT(z)$ and set
\begin{align*}
  \projSZ_\TT v = \sum_{z\in\NN_\TT^\Omega} \alpha_{\TT,z} \eta_{\TT,z} :=  \sum_{z\in\NN_\TT^\Omega} \sum_{T\in\gamma_z} \frac{|T|\ip{v}{\psi_{T,z}}_T}{\sum_{T'\in\gamma_z} |T'|} \eta_{\TT,z},
\end{align*}
where $\psi_{T,z}\in\PP^1(T)$ is the unique element with $\ip{\psi_{T,z}}{\eta_{\TT,z'}}_T = \delta_{z,z'}$ for all $z,z'\in \NN(T)$. 
Some comments are in order:
\begin{remark}\label{rem:SZ}
  It is common to define the quasi-interpolation operator $\projSZ_\TT$ with $\gamma_z$ containing exactly one element.
  The case $\gamma_z = \patch_\TT(z)$ has been used in various works, cf.~\cite{HiptmairWuZheng2012,StevensonVanVenetie2020negOrder,WuZheng2017}.
  We note that the authors of~\cite{WuZheng2017} define the operator with coefficients
  \begin{align*}
    \alpha_{\TT,z} = \sum_{T\in\patch_\TT(z)} \frac{|T|\Pi_T^1 v(z)}{\sum_{T'\in\patch_\TT(z)} |T'|},
  \end{align*}
  where $\Pi_T^1$ denotes the $L^2(T)$ orthogonal projection on $\PP^1(T)$.
  This definition is identical to the operator above with $\gamma_z = \patch_\TT(z)$ which can be seen from the identity
    \begin{align*}
      \ip{v}{\psi_{T,z}}_T = \ip{v}{\Pi_T^1 \psi_{T,z}}_T = \ip{\Pi_T^1 v}{\psi_{T,z}}_T = (\Pi_T^1v)(z).
    \end{align*}
    Explicit representations of the coefficients are obtained by straightforward computations and symmetry arguments (see~\cite[Remark~3.3]{StevensonVanVenetie2020negOrder} or~\cite{HiptmairWuZheng2012} for the case $d=3$). 
    Extending the functions $\varphi_{T,z}$ by $0$ one can show that (denoting with $\Gamma_z$ the domain associated to $\gamma_z$)
    \begin{align*}
      \sum_{T\in\gamma_z} \frac{|T|}{\sum_{T'\in\gamma_z}|T'|} \psi_{T,z}|_{\Gamma_z} = \frac{1}{|\Gamma_z|}(\alpha\eta_{\TT,z}+\beta)|_{\Gamma_z} \in \PP^1(\gamma_z)
    \end{align*}
    The constants $\alpha$,$\beta$ only depend on the space dimension $d\geq 1$ but are independent of $z$ and $\gamma_z$.
\end{remark}
For the remainder of this work we use $\gamma_z = \patch_\TT(z)$ in the definition of $\projSZ_\TT$. 
The next lemma collects some well-known results on $\projSZ_\TT$ (see~\cite[Lemma~3.2--3.6]{WuZheng2017}):
\begin{lemma}\label{lem:SZ}
  The operator $\projSZ_\TT\colon L^2(\Omega)\to \cS_0^1(\TT)$ satisfies:
  \begin{enumerate}[label=(\alph*)]
    \item \textbf{Projection:} $\projSZ_\TT^2 = \projSZ_\TT$.        
    \item \textbf{Local boundedness:} 
      $\norm{\projSZ_\TT v}{T} \lesssim \norm{v}{\Patch_\TT(T)}$ resp. 
      $\norm{\nabla\projSZ_\TT v}{T} \lesssim \norm{\nabla v}{\Patch_\TT(T)}$ 
      \\
      for all $v\in L^2(\Omega)$ resp. $v\in H_0^1(\Omega)$ and $T\in\TT$.
    \item \textbf{Local approximation:}
        $\norm{v-\projSZ_\TT v}T \lesssim h_T \norm{\nabla v}{\Patch_\TT(T)}$ for all $v\in H_0^1(\Omega), T\in\TT$.
  \end{enumerate}
  The involved constants only depend on $d$ and shape regularity of $\TT$.
\end{lemma}

Let $\cS^{b}(\TT)\subset H_0^1(\Omega)$ denote the space of bubble functions with basis functions $\eta_{b,T} = \prod_{z\in\NN(T)} \eta_{\TT,z}$ for $T\in\TT$
and set $\cS_{0}^{1,b} = \cS_0^1(\TT) \oplus \cS^b(\TT)$.
The operator defined in the next result is a Fortin-type operator (below this is called orthogonality property).
\begin{lemma}\label{lem:projHone}
  The operator $\projHone_\TT \colon L^2(\Omega) \to \cS_0^{1,b}(\TT)$ defined by
  \begin{align*}
    \projHone_\TT v := \projSZ_\TT v + \projBubble_\TT (1-\projSZ_\TT)v 
    := \projSZ_\TT v + \sum_{T\in\TT} \frac{\ip{v-\projSZ_\TT v}1_T}{\ip{\eta_{b,T}}1_T} \eta_{b,T} 
  \end{align*}
  has the following properties:
  \begin{enumerate}[label=(\alph*)]
    \item \label{prop:projHone:proj} \textbf{Quasi-projection:}
        $\projHone_\TT v = v \quad\text{for all } v\in \cS_0^{1}(\TT)$. 
    \item\label{prop:projHone:orth} \textbf{Orthogonality:}
      $\ip{(1-\projHone_\TT)v}1_T = 0$ for all $T\in\TT$ and $v\in L^2(\Omega)$.
    \item\label{prop:projHone:app} \textbf{Local approximation property:} 
      \begin{align*}
        \norm{(1-\projHone_\TT)v}{T} \lesssim h_T \norm{\nabla v}{\Patch_\TT(T)} \quad\text{for all } T\in\TT \text{ and } v\in H_0^1(\Omega).
      \end{align*}
    \item\label{prop:projHone:bound} \textbf{Locally bounded:}
      \begin{alignat*}{2}
        \norm{\projHone_\TT v}{T} &\lesssim \norm{v}{\Patch_\TT(T)} 
        &\quad&\text{for all } T\in\TT \text{ and }v\in L^2(\Omega), \\
        \norm{\nabla \projHone_\TT v}{T} &\lesssim \norm{\nabla v}{\Patch_\TT(T)} &\quad&\text{for all } T\in\TT \text{ and }v\in H_0^1(\Omega).
      \end{alignat*}
  \end{enumerate}
  The involved constants only depend on $d$ and shape regularity of $\TT$.
\end{lemma}
\begin{proof}
  \textbf{Proof of~\ref{prop:projHone:proj}.} 
  Using the projection property, i.e., $\projSZ_\TT v= v$ for $v\in \cS_0^1(\TT)$, we infer that
  \begin{align*}
    \projHone_\TT v = \projSZ_\TT v + \projBubble_\TT(1-\projSZ_\TT)v
    = \projSZ_\TT v = v \quad\text{for all } v\in \cS_0^1(\TT).
  \end{align*}

  \noindent
  \textbf{Proof of~\ref{prop:projHone:orth}.}
  This follows from the construction of the operator, i.e., 
  \begin{align*}
    \ip{\projHone_\TT v}1_T = \ip{\projSZ_\TT v}{1}_T + \frac{\ip{v-\projSZ_\TT v}1_T}{\ip{\eta_{b,T}}1_T} \ip{\eta_{b,T}}1_T = \ip{v}1_T
    \quad\text{for all } v\in L^2(\Omega).
  \end{align*}

  \noindent
  \textbf{Proof of~\ref{prop:projHone:app}.}
  With the local approximation property $\norm{v-\projSZ_\TT v}{T} \lesssim h_T \norm{\nabla v}{\Patch_\TT(T)}$ (Lemma~\ref{lem:SZ}) and $\norm{\eta_{b,T}}T \simeq |T|^{1/2} \simeq |T|^{-1/2} \ip{\eta_{b,T}}1_T$  we infer that
  \begin{align*}
    \norm{v-\projHone_\TT v}T \leq \norm{v-\projSZ_\TT v}T + \norm{v-\projSZ_\TT v}T \frac{|T|^{1/2}}{\ip{\eta_{b,T}}1_T} \norm{\eta_{b,T}}T \simeq \norm{v-\projSZ_\TT v}T \lesssim h_T \norm{\nabla v}{\Patch_\TT(T)}.
  \end{align*}

  \noindent
  \textbf{Proof of~\ref{prop:projHone:bound}.}
  The local $L^2$ bound for the operator $\projSZ_\TT$ (Lemma~\ref{lem:SZ}) together with the scaling arguments that we have used before implies
  \begin{align*}
    \norm{\projHone_\TT v}T \lesssim \norm{\projSZ_\TT v}T + \norm{(1-\projSZ_\TT)v}T \lesssim \norm{v}{\Patch_\TT(T)}.
  \end{align*}
  Finally, the local $H^1$ bound follows from a corresponding bound for $\projSZ_\TT$, scaling arguments, and the local approximation property, i.e.
  \begin{align*}
    \norm{\nabla \projHone_\TT v}T &\leq \norm{\nabla \projSZ_\TT v}T + \norm{(1-\projSZ_\TT)v}T \frac{|T|^{1/2}}{\ip{\eta_{b,T}}1_T} \norm{\nabla \eta_{b,T}}T 
    \\ &\simeq \norm{\nabla \projSZ_\TT v}T + h_T^{-1} \norm{(1-\projSZ_\TT)v}T \lesssim \norm{\nabla v}{\Patch_\TT(T)},
  \end{align*}
  which concludes the proof.
\end{proof}

With $\projHone_\TT' \colon L^2(\Omega)\to L^2(\Omega)$ we denote the adjoint of $\projHone_\TT$.
Writing $\projHone_\TT = \projSZ_\TT + \projBubble_\TT(1-\projHone_\TT)$ we see
\begin{align*}
  \projHone_\TT' = \projSZ_\TT' + (1-\projSZ_\TT')\projBubble_\TT',
\end{align*}
where
\begin{align*}
  \projSZ_\TT'v &= \sum_{z\in\NN_\TT} \ip{v}{\eta_{\TT,z}} \varphi_{\TT,z} \quad\text{and}\quad
  \projBubble_\TT' v= \sum_{T\in\TT}  \frac{\ip{v}{\eta_{b,T}}_T}{\ip{\eta_{b,T}}1_T} \chi_T.
\end{align*}
Note that by Remark~\ref{rem:SZ} we have that $\varphi_{\TT,z} = 1/|\Omega_\TT(z)|(\alpha \eta_{\TT,z}+\beta)$ on $\Omega_\TT(z)$ and $0$ otherwise.
From this explicit representation it is thus straightforward to see that $\projHone_\TT' \phi\in \PP^1(\TT)$ for all $\phi\in L^2(\Omega)$. We will use this fact in order to apply (local) inverse estimates, e.g. $\norm{h_\TT\projHone_\TT'\phi}{} \lesssim \norm{\projHone_\TT'\phi}{-1}$.
Moreover, note that $\projHone_\TT$ is bounded in $H_0^1(\Omega)$ so that $\projHone_\TT'\colon H^{-1}(\Omega)\to \PP^1(\TT)$ is well-defined.

The next result follows more or less immediately from the properties described in Lemma~\ref{lem:projHone}. 
\begin{lemma}\label{lem:projHoneDual}
  The operator $\projHone_\TT'$ satisfies the following properties:
  \begin{enumerate}[label=(\alph*)]
    \item \textbf{Quasi-projection:}\label{prop:projHoneDual:proj}
        $\projHone_\TT' \phi = \phi$ for all $\phi \in \PP^0(\TT)$.
    \item \textbf{Approximation:}\label{prop:projHoneDual:app}
        $\norm{(1-\projHone_\TT')\phi}{-1} \lesssim \norm{h_\TT\phi}{}$ for all $\phi\in L^2(\Omega)$.
    \item \textbf{Local boundedness for $L^2(\Omega)$ functions:}\label{prop:projHoneDual:bound}
      \begin{align*}
        \norm{\projHone_\TT'\phi}{T} \lesssim \norm{\phi}{\Patch_\TT(T)} \quad\text{and}\quad
        \norm{\projHone_\TT'\phi}{-1,T} \lesssim \norm{\phi}{-1,\Patch_\TT(T)} \quad\text{for all } T\in\TT \text{ and }\phi \in L^2(\Omega).
      \end{align*}
    \item \textbf{Global boundedness:}\label{prop:projHoneDual:bound:global}
        $\norm{\projHone_\TT'\phi}{-1} \lesssim \norm{\phi}{-1}$ for all $\phi \in H^{-1}(\Omega)$.
  \end{enumerate}
  The involved constants only depend on $d$ and shape regularity of $\TT$.
\end{lemma}
\begin{proof}
  \textbf{Proof of~\ref{prop:projHoneDual:proj}.}
  The quasi-projection property can be seen from the orthogonality property $\ip{(1-\projHone_\TT)v}1_T = 0$ yielding the identity
  \begin{align*}
    \ip{\projHone_\TT'\phi}v = \ip{\phi}{\projHone_\TT v} = \ip{\phi}{v} \quad\text{for all } \phi\in \PP^0(\TT), \, v\in H_0^1(\Omega). 
  \end{align*}

  \noindent
  \textbf{Proof of~\ref{prop:projHoneDual:app}.}
  Let $\phi\in L^2(\Omega)$ and $T\in\TT_\ell$. 
  Using the approximation property of $\projHone_\TT$ we get that
  \begin{align*}
    \ip{(1-\projHone_\TT')\phi}v = \ip{\phi}{(1-\projHone_\TT)v} \lesssim \norm{h_\TT\phi}{} \norm{\nabla v}{}.
  \end{align*}
  Dividing by $\norm{\nabla v}{}$ and taking the supremum over all $0\neq v\in H_0^1(\Omega)$ shows the assertion. 

  \noindent
  \textbf{Proof of~\ref{prop:projHoneDual:bound}.}
  Local boundedness in $L^2$ can be derived from the local definition of the operator.

  To see local boundedness in $H^{-1}$ we extend any $v\in H_0^1(T)$ by $0$ on $\Omega\setminus  T$.
  Then, using duality, continuity of $\projHone_\TT$ and $\supp \projHone_\TT v \subseteq \overline\Patch_\TT(T)$ we conclude that
  \begin{align*}
    \ip{\projHone_\TT'\phi}{v}_T &= \ip{\phi}{\projHone_\TT v}_{\Patch_\TT(T)} 
    \lesssim \norm{\phi}{-1,\Patch_\TT(T)}\norm{\nabla \projHone_\TT v}{\Patch_\TT(T)} \lesssim \norm{\phi}{-1,\Patch_\TT(T)} \norm{\nabla v}{T}.
  \end{align*}

  \noindent
  \textbf{Proof of~\ref{prop:projHoneDual:bound:global}.}
  This follows directly from duality and global boundedness of $\projHone_\TT$ in $H_0^1(\Omega)$. 
\end{proof}

Recall that the range of $\projHone_\TT'$ is a subspace of $\PP^1(\TT)$. In order to obtain a projection onto piecewise constants we apply the $L^2(\Omega)$ orthogonal projection.
\begin{theorem}\label{thm:projHmOneSZ}
  The operator $\projHmOneSZ_\TT := \projLtwo_\TT \projHone_\TT'$ has the following properties:
  \begin{enumerate}[label=(\alph*)]
    \item\label{prop:projHoneSZ:proj} \textbf{Projection:}
      $\projHmOneSZ_\TT^2 = \projHmOneSZ_\TT$.
    \item\label{prop:projHoneSZ:app} \textbf{Approximation:}
        $\norm{(1-\projHmOneSZ_\TT)\phi}{-1} \lesssim \norm{h_\TT \phi}{}$ for all $\phi\in L^2(\Omega)$.
    \item\label{prop:projHoneSZ:bound}\textbf{Local boundedness for $L^2(\Omega)$ functions:}
      \begin{align*}
        \norm{\projHmOneSZ_\TT\phi}{T} \lesssim \norm{\phi}{\Patch_\TT(T)} \quad\text{and}\quad
        \norm{\projHmOneSZ_\TT\phi}{-1,T} \lesssim \norm{\phi}{-1,\Patch_\TT(T)}
        \quad\text{for all }T\in\TT \text{ and }\phi\in L^2(\Omega).
      \end{align*}
    \item\label{prop:projHoneSZ:bound:global} \textbf{Global boundedness:}
        $\norm{\projHmOneSZ_\TT\phi}{-1} \lesssim \norm{\phi}{-1}$ for all $\phi\in H^{-1}(\Omega)$.
    \end{enumerate}
    The involved constants only depend on $d$ and shape regularity of $\TT$.
\end{theorem}
\begin{proof}
  \textbf{Proof of~\ref{prop:projHoneSZ:proj}.}
  The projection property follows from $\projHone_\TT' \phi = \phi$ for $\phi\in\PP^0(\TT)$ (Lemma~\ref{lem:projHoneDual}), which yields
  \begin{align*}
    \projHmOneSZ_\TT^2 = \projLtwo_\TT \projHone_\TT' (\projLtwo_\TT \projHone_\TT') = \projLtwo_\TT (\projLtwo_\TT \projHone_\TT') = \projLtwo_\TT \projHone_\TT' = \projHmOneSZ_\TT.
  \end{align*}

  \noindent
  \textbf{Proof of~\ref{prop:projHoneSZ:app}.}
  For the approximation property we note that
  \begin{align}\label{eq:proof:HoneSZ:b}
    \norm{(1-\projHmOneSZ_\TT)\phi}{-1} \leq \norm{(1-\projLtwo_\TT)\projHone_\TT'\phi}{-1} + \norm{(1-\projHone_\TT')\phi}{-1}.
  \end{align}
  The first term is estimated with the approximation property of $\projLtwo_\TT$ and the local $L^2$ boundedness of $\projHone_\TT'$ (Lemma~\ref{lem:projHoneDual}), i.e., 
  \begin{align*}
    \norm{(1-\projLtwo_\TT)\projHone_\TT'\phi}{-1} \lesssim \norm{h_\TT \projHone_\TT'\phi}{} \lesssim \norm{h_\TT \phi}{}.
  \end{align*}
  The second term in~\eqref{eq:proof:HoneSZ:b} is estimated using the approximation property of $\projHone_\TT'$, see Lemma~\ref{lem:projHoneDual}.
  
  \noindent
  \textbf{Proof of~\ref{prop:projHoneSZ:bound}.}
  Note that $\projLtwo_\TT$ is locally bounded. Together with the local boundedness of $\projHone_\TT'$ (Lemma~\ref{lem:projHoneDual}) we get that
  \begin{align*}
    \norm{\projLtwo_\TT\projHone_\TT'\phi}T \leq \norm{\projHone_\TT'\phi}T \lesssim \norm{\phi}{\Patch_\TT(T)}.
  \end{align*}

  For the local bound in $H^{-1}(T)$ we stress that $\norm{\projLtwo_\TT\projHone_\TT'\phi}{-1,T}\lesssim \norm{\projHone_\TT'\phi}{-1,T}$ which can be seen from $\norm{v}T \lesssim h_T\norm{\nabla v}T$ for $v\in H_0^1(T)$ yielding
  \begin{align*}
    \ip{\projLtwo_\TT\projHone_\TT'\phi}v_T = \ip{\projHone_\TT'\phi}{\projLtwo_\TT v}_T \lesssim \norm{\projHone_\TT'\phi}{T}\norm{v}T \lesssim 
    h_T\norm{\projHone_\TT'\phi}{T}\norm{\nabla v}T.
  \end{align*}
  The assertion then follows from the inverse inequality~\eqref{eq:invest} and local boundedness of $\projHone_\TT'$ (Lemma~\ref{lem:projHoneDual}).

  \noindent
  \textbf{Proof of~\ref{prop:projHoneSZ:bound:global}.}
  Boundedness in $H^{-1}(\Omega)$  is not directly clear due to $\projLtwo_\TT$. 
  However, using the approximation property of $\projLtwo_\TT$ we get that $\norm{\projLtwo_\TT\psi}{-1} \lesssim \norm{\psi}{-1} + \norm{h_\TT \psi}{}$ and by applying the inverse inequality~\eqref{eq:invest} we further conclude that 
  \begin{align*}
    \norm{\projLtwo_\TT\psi}{-1} \lesssim \norm{\psi}{-1} \quad\text{for all } \psi \in \PP^1(\TT).
  \end{align*}
  That is, $\projLtwo_\TT$ is bounded in $H^{-1}(\Omega)$ when restricted to $\PP^1(\TT)$. 
  Using the latter estimate with $\psi = \projHone_\TT'\phi$ and the boundedness of $\projHone_\TT'$ (Lemma~\ref{lem:projHoneDual}) implies that
  \begin{align*}
    \norm{\projHmOneSZ_\TT\phi}{-1} = \norm{\projLtwo_\TT\projHone_\TT'\phi}{-1} \lesssim \norm{\phi}{-1},
  \end{align*}
  which concludes the proof.
\end{proof}

\subsection{A local projection operator in $\widetilde H^{-1}(\Omega)$}\label{sec:projHmOneSZtilde}
We follow the same ideas as in Section~\ref{sec:projHmOneSZ} but only point out the differences in the definition. 
The results below follow the same argumentations as in Section~\ref{sec:projHmOneSZ} and are therefore omitted. 
For each $z\in\NN_\TT$ let $\emptyset\neq \gamma_z\subseteq\patch_\TT(z)$ and set
\begin{align*}
 \overline \projSZ_\TT v = \sum_{z\in\NN_\TT} \alpha_{\TT,z} \eta_{\TT,z} :=  \sum_{z\in\NN_\TT} \sum_{T\in\gamma_z} \frac{|T|\ip{v}{\psi_{T,z}}_T}{\sum_{T'\in\gamma_z} |T'|} \eta_{\TT,z}.
\end{align*}
Throughout, we consider $\gamma_z = \patch_\TT(z)$, see Remark~\ref{rem:SZ}. We define $\overline\projHone_\TT := \overline\projSZ_\TT + \projBubble_\TT(1-\overline\projSZ_\TT)$.

Before we state the results let us note that the local boundedness property Lemma~\ref{lem:projHoneDual}\ref{prop:projHoneDual:bound} does not hold for the $\norm\cdot{-1,T,\sim}$ norms due to different scaling properties. 
Nevertheless, we use an auxiliary norm defined with $H_{\Gamma,c}^1(\omega) = \set{v\in H^1(\omega)}{v|_{\partial\omega\setminus\Gamma} = 0}$ as
\begin{align*}
  \norm{\phi}{-1,\omega,\Gamma} := \begin{cases}
    \norm{\phi}{-1,\sim} & \text{if } |\omega|=|\Omega|, \\
    \sup_{0\neq v\in H_{\Gamma,c}^1(\Omega)} \frac{\dual{\phi}{v}_\omega}{\norm{\nabla v}{\omega}} & \text{else}.
  \end{cases}
\end{align*}
Note that if $|\partial\omega\cap\Gamma|=0$ then $H_{\Gamma,c}^1(\omega) = H_0^1(\omega)$, thus, $\norm{\phi}{-1,\omega,\Gamma}=\norm{\phi}{-1,\omega}$. 
In particular, we stress that local scaling arguments prove the inverse estimate
\begin{align*}
  h_T \norm{\phi}{T} \lesssim \norm{\phi}{-1,T,\Gamma} \quad\text{for all } T\in\TT \text{ and }\phi\in \PP^p(\TT).
\end{align*}

\begin{lemma}\label{lem:projTildeHoneDual}
  The operator $\overline\projHone_\TT'$ satisfies the following properties:
  \begin{enumerate}[label=(\alph*)]
    \item \textbf{Quasi-projection:}\label{prop:projTildeHoneDual:proj}
        $\overline\projHone_\TT' \phi = \phi$ for all $\phi \in \PP^0(\TT)$.
    \item \textbf{Approximation:}\label{prop:projTildeHoneDual:app}
        $\norm{(1-\overline\projHone_\TT')\phi}{-1,\sim} \lesssim \norm{h_\TT\phi}{}$ for all $\phi\in L^2(\Omega)$. 
    \item \textbf{Local boundedness for $L^2(\Omega)$ functions:}\label{prop:projTildeHoneDual:bound}
      \begin{align*}
        \norm{\overline\projHone_\TT'\phi}{T} \lesssim \norm{\phi}{\Patch_\TT(T)} \quad\text{and}\quad
        \norm{\overline\projHone_\TT'\phi}{-1,T,\Gamma} \lesssim \norm{\phi}{-1,\Patch_\TT(T),\Gamma} \quad\text{for all } T\in\TT \text{ and }\phi \in L^2(\Omega).
      \end{align*}
    \item \textbf{Global boundedness:}\label{prop:projTildeHoneDual:bound:global}
        $\norm{\overline\projHone_\TT'\phi}{-1,\sim} \lesssim \norm{\phi}{-1,\sim}$ for all  $\phi \in\widetilde H^{-1}(\Omega)$.
  \end{enumerate}
    The involved constants only depend on $d$ and shape regularity of $\TT$.
\end{lemma}

\begin{theorem}\label{thm:projTildeSZ}
  The operator $\widetilde\projHmOneSZ_\TT := \projLtwo_\TT \overline\projHone_\TT'$ has the following properties:
  \begin{enumerate}[label=(\alph*)]
    \item\label{prop:projTildeHoneSZ:proj} \textbf{Projection:}
        $\widetilde\projHmOneSZ_\TT^2 = \widetilde\projHmOneSZ_\TT$.
    \item\label{prop:projTildeHoneSZ:app} \textbf{Approximation:}
        $\norm{(1-\widetilde\projHmOneSZ_\TT)\phi}{-1,\sim} \lesssim \norm{h_\TT \phi}{}$ for all $\phi\in L^2(\Omega)$.
    \item\label{prop:projTildeHoneSZ:bound}\textbf{Local boundedness for $L^2(\Omega)$ functions:}
      \begin{align*}
        \norm{\widetilde\projHmOneSZ_\TT\phi}{T} \lesssim \norm{\phi}{\Patch_\TT(T)} \quad\text{and}\quad
        \norm{\widetilde\projHmOneSZ_\TT\phi}{-1,T,\Gamma} \lesssim \norm{\phi}{-1,\Patch_\TT(T),\Gamma}
        \quad\text{for all } T\in\TT \text{ and } \phi\in L^2(\Omega).
      \end{align*}
    \item\label{prop:projTildeHoneSZ:bound:global} \textbf{Global boundedness:}
        $\norm{\widetilde\projHmOneSZ_\TT\phi}{-1,\sim} \lesssim \norm{\phi}{-1,\sim}$ for all $\phi\in \widetilde H^{-1}(\Omega)$.
    \end{enumerate}
    The involved constants only depend on $d$ and shape regularity of $\TT$.
\end{theorem}

\subsection{Equivalent norms in interpolation spaces}\label{sec:normsIntSpaces}
In this section we collect some results on the relation between interpolation and approximation spaces.
Recall the definition of a sequence of uniform meshes from Section~\ref{sec:mesh}.

\begin{lemma}\label{lem:decompositionUniform}
  Let $s\in(0,1)$, $p\in\N_0$ and let $(\TT_\ell)_{\ell=0}^L$ denote a sequence of uniform meshes with mesh sizes $(h_\ell)_{\ell=0}^L$.
  Let $\phi = \sum_{\ell=0}^L\phi_\ell$ with $\phi_\ell\in \PP^p(\TT_\ell)$. 
  Then,
  \begin{align*}
    \norm{\phi}{-s}^2 \lesssim \sum_{\ell=0}^L h_\ell^{-2+2s} \norm{\phi_\ell}{-1}^2 \quad\text{and}\quad
    \norm{\phi}{-s,\sim}^2 \lesssim \sum_{\ell=0}^L h_\ell^{-2+2s} \norm{\phi_\ell}{-1,\sim}^2.
  \end{align*}
  The involved constants only depend on $s$, $\Omega$, $d$, $p$ and the constants from~\ref{ass:mesh:reg}--\ref{ass:mesh:gen}.
\end{lemma}
\begin{proof}
  Fix some $\delta>0$ with $-1<-s-\delta<-s+\delta<0$.
  The $H^{-s}(\Omega)$ scalar product is denoted with $\ip\cdot\cdot_{-s}$. 
  Then, the reiteration theorem (\cite[Theorem~2.2]{CWHM15}), interpolation estimates in Hilbert spaces  and inverse inequalities show that
  \begin{align*}
    \norm{\phi}{-s}^2 = \sum_{\ell=0}^L \sum_{k=0}^L \ip{\phi_\ell}{\phi_k}_{-s} 
    \lesssim \sum_{\ell=0}^L \sum_{k=\ell}^L |\ip{\phi_\ell}{\phi_k}_{-s}|
    &\leq \sum_{\ell=0}^L \sum_{k=\ell}^L \norm{\phi_\ell}{-s+\delta} \norm{\phi_k}{-s-\delta} \\
    & \lesssim \sum_{\ell=0}^L \sum_{k=\ell}^L h_\ell^{-1+s-\delta} \norm{\phi_\ell}{-1} h_k^{-1+s+\delta}\norm{\phi_k}{-1} \\
    & \leq \norm{A}2 \sum_{\ell=0}^L h_\ell^{2(-1+s)} \norm{\phi_\ell}{-1}^2,
  \end{align*}
  where the matrix $A\in \R^{(L+1)\times (L+1)}$ is defined by $A_{\ell k} := h_\ell^{-\delta} h_k^{\delta}$ for $k\geq \ell$ and $A_{\ell k} := 0$ for $k<\ell$.
  From Lemma~\ref{lem:meshgenunif} we get that $h_\ell \simeq q_\mathrm{ref}^{\ell}$, thus, we conclude that $\norm{A}2\lesssim 1$.
  The second assertion is proved following the same lines of argumentation.
\end{proof}

The next result presented will be a key ingredient in the stability analysis of multilevel decompositions. It shows the deep connection between approximation spaces and interpolation spaces.
Such results are known and in fact the reason for the optimality of the BPX preconditioner. We refer to~\cite{Bornemann94} for an overview and a short discussion as well as further references. 
A proof is included for the convenience of the reader where we follow the same argumentation as in~\cite[Theorem~1]{Bornemann94}.
We note that the main argumentation is the use of the inverse inequality and the approximation property. One also uses the discrete version of the $K$ and $J$ method, see e.g.~\cite{PeetreBook}.
Moreover, note that for a sequence of uniform meshes $\bigcup_{\ell\in\N_0} \PP^0(\TT_\ell)$ is dense in $H^{-s}(\Omega)$ and $\widetilde H^{-s}(\Omega)$. 

\begin{theorem}\label{thm:normequivInterpolation}
  Let $(\TT_\ell)_{\ell\in\N_0}$ denote a sequence of uniform meshes with mesh-sizes $(h_\ell)_{\ell\in\N_0}$.
  Let $s\in(0,1)$. The norm equivalences
  \begin{align}
    \norm{\phi}{-s}^2 \simeq \norm{\phi}{-1}^2 + \sum_{\ell=0}^\infty h_\ell^{-2+2s} \norm{\phi-\projHmOne_\ell\phi}{-1}^2 
    \quad\text{for all }\phi\in H^{-s}(\Omega)
  \end{align}
  and
  \begin{align}
    \norm{\phi}{-s,\sim}^2 \simeq \norm{\phi}{-1,\sim}^2 + \sum_{\ell=0}^\infty h_\ell^{-2+2s} \norm{\phi-\projTildeHmOne_\ell\phi}{-1,\sim}^2 
    \quad\text{for all }\phi\in \widetilde H^{-s}(\Omega)
  \end{align}
  hold true and the involved constants only depend on $s$, $\Omega$, $d$ and the constants from~\ref{ass:mesh:reg}--\ref{ass:mesh:gen}. 
\end{theorem}
\begin{proof}
  Let $0<s<1$. We consider the discrete $K$ and $J$ version of interpolation, see, e.g.~\cite[Section~1.7]{Triebel}, 
  where $K(t,\phi) = \inf_{v\in L^2(\Omega)}(\norm{\phi-v}{-1} + t \norm{v}{})$ resp. $J(t,\phi_k) := \max\{\norm{\phi_k}{-1},t\norm{\phi_k}{}\}$ denote the $K$ resp. $J$ functional.
  By the equivalence of interpolation norms (\cite[Theorem~3.3.1]{BerghLoefstroem76}) and the equivalence of discrete interpolation norms (\cite[Section~1.7]{Triebel}) we conclude that
  \begin{align*}
    \norm{\phi}{-s}^2 \simeq \sum_{k\in\mathbb{Z}} \big( q_\mathrm{ref}^{k(-1+s)} K(q_\mathrm{ref}^k,\phi) \big)^2
    \simeq \inf\set{\sum_{k\in\mathbb{Z}} \big( q_\mathrm{ref}^{k(-1+s)} J(q_\mathrm{ref}^k,\phi_k) \big)^2}{\phi=\sum_{k\in\mathbb{Z}} \phi_k, \, \phi_k\in L^2(\Omega)},
  \end{align*}
where the infimum is taken over all possible decompositions. We note that the equivalence results in~\cite[Section~1.7]{Triebel} are given for $q_\mathrm{ref}=1/2$ but hold true for $0<q_\mathrm{ref}<1$ following the same proofs.

  Note that $\phi = \sum_{\ell=0}^\infty (\projHmOne_\ell-\projHmOne_{\ell-1}) \phi =: \sum_{\ell=0}^\infty \phi_\ell$.
  The inverse estimate, the equivalence $h_\ell \simeq q_\mathrm{ref}^\ell$ (Lemma~\ref{lem:meshgenunif}) and boundedness of the projection operators imply that
  \begin{align*}
    J(q_\mathrm{ref}^\ell,\phi_\ell) \lesssim \norm{\phi_\ell}{-1} = \norm{(\projHmOne_\ell-\projHmOne_{\ell-1}) \phi}{-1}
    \leq \norm{(1-\projHmOne_{\ell-1})\phi}{-1}.
  \end{align*}
  Therefore,
  \begin{align*}
    \norm{\phi}{-s}^2 \lesssim \sum_{\ell=0}^\infty \big(q_\mathrm{ref}^{\ell(-1+s)}J(q_\mathrm{ref}^\ell,\phi_\ell)\big)^2
    \lesssim \norm{\phi}{-1}^2 + \sum_{\ell=0}^\infty h_\ell^{-2+2s} \norm{(1-\projHmOne_\ell)\phi}{-1}^2.
  \end{align*}

  To see the other direction we use the approximation property and $h_\ell\simeq q_\mathrm{ref}^\ell$ to get that
  \begin{align*}
    \norm{\phi-\projHmOne_\ell\phi}{-1} &= \inf_{v\in L^2(\Omega)} \norm{\phi-\projHmOne_\ell v}{-1} \leq
    \inf_{v\in L^2(\Omega)} \big( \norm{\phi-v}{-1} + \norm{v-\projHmOne_\ell v}{-1} \big)
    \\ 
    &\lesssim \inf_{v\in L^2(\Omega)} \big( \norm{\phi-v}{-1} + h_\ell \norm{v}{} \big)
    \simeq \inf_{v\in L^2(\Omega)} \big( \norm{\phi-v}{-1} + q_\mathrm{ref}^\ell \norm{v}{} \big)
    = K(q_\mathrm{ref}^\ell,\phi).
  \end{align*}
  Then,
  \begin{align*}
    \sum_{\ell=0}^\infty h_\ell^{-2+2s}\norm{(1-\projHmOne_\ell)\phi}{-1}^2 
    \lesssim \sum_{\ell=0}^\infty \big(q_\mathrm{ref}^{\ell(-1+s)} K(q_\mathrm{ref}^\ell,\phi) \big)^2 \lesssim \norm{\phi}{-s}^2
  \end{align*}
  together with $\norm{\phi}{-1}\lesssim \norm{\phi}{-s}$ finishes the proof of the first equivalence.

  The second equivalence follows the same argumentation with obvious modifications. 
\end{proof}

As a consequence of the latter result we obtain a multilevel norm on $H^{-s}(\Omega)$ resp. $\widetilde H^{-s}(\Omega)$:
\begin{corollary}\label{cor:normequiv}
  Let $(\TT_\ell)_{\ell=0}^L$ denote a sequence of uniform refined meshes with mesh-size functions $(h_\ell)_{\ell=0}^L$.
  Let $s\in(0,1)$, then,
  \begin{align*}
    \norm{\phi}{-s}^2 &\simeq \sum_{\ell=0}^L h_\ell^{-2+2s} \norm{(\projHmOne_\ell-\projHmOne_{\ell-1})\phi}{-1}^2, \quad 
    \norm{\phi}{-s,\sim}^2 \simeq \sum_{\ell=0}^L h_\ell^{-2+2s} \norm{(\projTildeHmOne_\ell-\projTildeHmOne_{\ell-1})\phi}{-1,\sim}^2
  \end{align*}
  for all $\phi \in \PP^0(\TT_L)$. 
  The constants only depend on $s$, $\Omega$, $d$ and the constants from~\ref{ass:mesh:reg}--\ref{ass:mesh:gen}.
\end{corollary}
\begin{proof}
  From the properties of the projection operators and applying Theorem~\ref{thm:normequivInterpolation} we get that
  \begin{align*}
    \sum_{\ell=0}^L h_\ell^{-2+2s} \norm{(\projHmOne_\ell-\projHmOne_{\ell-1})\phi}{-1}^2
    \leq \sum_{\ell=0}^L h_\ell^{-2+2s} \norm{(1-\projHmOne_{\ell-1})\phi}{-1}^2 \lesssim \norm{\phi}{-s}^2.
  \end{align*}
  To see the other direction we apply Lemma~\ref{lem:decompositionUniform} with $\phi_\ell = (\projHmOne_\ell-\projHmOne_{\ell-1})\phi$ and note that $\sum_{\ell=0}^L \phi_\ell = \projHmOne_L\phi = \phi$.  
  The second equivalence follows the same lines of proof and is therefore omitted. 
\end{proof}

\subsection{Additive Schwarz norms}\label{sec:abstract}
Let $\HHH$ denote a Hilbert space with norm $\norm\cdot{\HHH}$ and let $\HHH_h\subset \HHH$ denote a finite-dimensional
subspace.
Let $\XX_i\subset\HHH_h$, $i\in\II$ with $\#\II<\infty$ and set
\begin{align*}
  \XX = \set{\XX_i}{i\in\II}.
\end{align*}

We say $\XX$ is a decomposition of $\HHH_h$ if
\begin{align*}
  \HHH_h = \sum_{i\in\II} \XX_i.
\end{align*}
To the decomposition $\XX$ we associate the \emph{additive Schwarz norm} $\enorm\cdot_{\HHH,\XX}$ given by
\begin{align*}
  \enorm{x}_{\HHH,\XX}^2 := \inf\set{\sum_{i\in\II} \norm{x_i}{\HHH}^2}{x_i\in \XX_i \text{ such that } x = \sum_{i\in\II} x_i}.
\end{align*}

A key ingredient in most works on additive Schwarz methods is to establish norm equivalence of the form 
\begin{align}\label{eq:equiv:abstract}
  C_1 \enorm{x}_{\HHH,\XX} \leq \norm{x}{\HHH} \leq C_2 \enorm{x}_{\HHH,\XX} \quad\text{for all }x\in \HHH_h.
\end{align}
This equivalence implies that the associated additive Schwarz preconditioner yields preconditioned systems with condition 
numbers depending only on $C_1,C_2$.
This is well-known in the context of additive Schwarz preconditioners
and we refer the interested reader to~\cite{oswald94,ToselliWidlund}.


Note that~\eqref{eq:equiv:abstract} is equivalent to the following two estimates: 
\begin{itemize}
  \item The lower bound in~\eqref{eq:equiv:abstract} is equivalent to: For every $x\in\HHH_h$ there exist $x_i\in \XX_i$, $i\in\II$ such that $x=\sum_{i\in\II} \XX_i$ and 
    \begin{align}\label{eq:observationLowerBound}
      \sum_{i\in\II} \norm{x_i}{\HHH}^2 \leq C_1^{-2} \norm{x}{\HHH}^2.
    \end{align}
  \item The upper bound in~\eqref{eq:equiv:abstract} is equivalent to: For all $x_i\in \XX_i$, $x = \sum_{i\in\II} x_i\in\HHH_h$,
    \begin{align}\label{eq:observationUpperBound}
      C_2^{-2}\norm{x}{\HHH}^2 \leq \sum_{i\in\II} \norm{x_i}{\HHH}^2.
    \end{align}
\end{itemize}
In both estimates the constants $C_1$, $C_2$ are independent of $x$, $x_i$.

The lower bound in~\eqref{eq:equiv:abstract} is called \emph{stability of the decomposition}, see, e.g.~\cite[Assumption~2.2]{ToselliWidlund}. 
The upper bound in~\eqref{eq:equiv:abstract} is often proved by showing \emph{strengthened Cauchy--Schwarz inequalities}, see, e.g.~\cite[Assumption~2.3]{ToselliWidlund}.
Another closely related method is the \emph{multiplicative Schwarz method}, see, e.g.~\cite[Section~2.2]{ToselliWidlund}.
Convergence results of the latter method usually involve the same analytical tools, namely, the stability of subspace decompositions and strengthened Cauchy--Schwarz inequalities, see, e.g.~\cite[Theorem~2.9]{ToselliWidlund}.

\subsection{Stable splittings in $H^{-1}(\Omega)$ and $\widetilde H^{-1}(\Omega)$}
In this section we recall some results of our work~\cite{FuehrerHeuerQuasiDiagonal19} on stable one-level decompositions
which will be key ingredients in the proof of our main theorems for uniform meshes.

For a mesh $\TT$ we consider the decomposition $\XX = \set{\XX_E}{E\in\EE}$ of $\PP^0(\TT)$, i.e., 
\begin{align*}
  \PP^0(\TT) = \sum_{E\in\EE} \XX_E,
\end{align*}
where $\XX_E = \linhull\{\psi_E\}$.
We use the notation $\enorm{\cdot}_{-1}=\enorm{\cdot}_{H^{-1}(\Omega),\XX}$ and recall the result~\cite[Theorem~3]{FuehrerHeuerQuasiDiagonal19}:
\begin{theorem}\label{thm:normEquivOnelevel}
  For $\phi\in \PP^0(\TT)$,
  \begin{align*}
    \enorm{\phi}_{-1} \lesssim \norm{\phi}{-1} \lesssim \enorm{\phi}_{-1},
  \end{align*}
  and the involved constants depend only on $\Omega$, $d$ and shape regularity of $\TT$.
\end{theorem}

A similar result is valid for the space $\widetilde H^{-1}(\Omega)$ with the decomposition $\widetilde\XX = \{\XX_\Omega\} \cup \set{\XX_E}{E\in\EE^\Omega}$ of $\PP^0(\TT)$, i.e.,
\begin{align*}
  \PP^0(\TT) = \XX_\Omega + \sum_{E\in\EE^\Omega} \XX_E,
\end{align*}
where $\XX_\Omega = \linhull\{1\}$.
Let us note that the latter splitting already implies the unique decomposition 
$\phi = \phi_0 + \phi_* =: \projLtwo_\Omega \phi + \phi_*$ for $\phi\in\PP^0(\TT)$, 
where $\projLtwo_\Omega$ denotes the $L^2(\Omega)$ projection to $\XX_\Omega = \PP^0(\Omega)$. Note that $\phi_*\in \PP_*^0(\TT):=\set{\psi\in\PP^0}{\projLtwo_\Omega\psi=0}$.
Moreover, we stress that
\begin{align*}
  \norm{\phi}{-1,\sim}^2 \simeq \norm{\phi_0}{-1,\sim}^2 + \norm{\phi_*}{-1,\sim}^2.
\end{align*}
We use the notation $\enorm{\cdot}_{-1,\sim} = \enorm{\cdot}_{\widetilde H^{-1}(\Omega),\widetilde\XX}$ and recall~\cite[Theorem~4]{FuehrerHeuerQuasiDiagonal19}:
\begin{theorem}\label{thm:normEquivOnelevel:tilde}
  For $\phi\in\PP^0(\TT)$,
  \begin{align*}
    \enorm{\phi}_{-1,\sim} \lesssim \norm{\phi}{-1,\sim} \lesssim \enorm{\phi}_{-1,\sim},
  \end{align*}
  and the involved constants depend only on $\Omega$, $d$ and shape regularity of $\TT$.
\end{theorem}

\begin{remark}
  We note that the proof of the lower bound in Theorem~\ref{thm:normEquivOnelevel} resp. Theorem~\ref{thm:normEquivOnelevel:tilde}
  depends on regularity results of the Poisson equation (Dirichlet problem resp. Neumann problem) in $\Omega$. 
  Regularity enters when estimating the $L^2(\Omega)$ norm of the Raviart--Thomas projection operator. 
  One can get rid of the dependence on regularity results by switching to a different projection operator that is $L^2(\Omega)$ stable up to oscillation terms, e.g., the operator from~\cite[Theorem~3.2]{egsv2019}, see also Lemma~\ref{lem:Hdivproj}.
\end{remark}

\section{Main results}\label{sec:main}
Throughout this section let $(\TT_\ell)_{\ell=0}^L$ denote a sequence of meshes.
For $\ell\geq 1$ define
\begin{alignat*}{2}
  \widetilde\EE_0 &:= \EE_0,  &\quad \widetilde\EE_\ell &:= \EE_\ell\setminus\EE_{\ell-1} \cup 
  \set{E\in\EE_\ell\cap \EE_{\ell-1}}{\supp(\psi_{\ell,E})\subsetneq \supp(\psi_{\ell-1,E})}, \\
  \widetilde\EE_0^\Omega &:= \EE_0^\Omega, &\quad 
  \widetilde\EE_\ell^\Omega &:= \EE_\ell^\Omega\setminus\EE_{\ell-1}^\Omega \cup 
  \set{E\in\EE_\ell^\Omega\cap \EE_{\ell-1}^\Omega}{\supp(\psi_{\ell,E})\subsetneq \supp(\psi_{\ell-1,E})}.
\end{alignat*}
Note that if $(\TT_\ell)_{\ell=0}^L$ is a sequence of uniform meshes, then $\EE_\ell = \widetilde\EE_\ell$ resp., $\EE_\ell^\Omega = \widetilde\EE_\ell^\Omega$.

\subsection{Multilevel decomposition for $H^{-s}(\Omega)$}
We consider the collection
\begin{align*}
  \XX_L = \set{\XX_{\ell,E}}{E\in\widetilde\EE_\ell, \, \ell=0,\dots,L},
  \quad\text{where }
  \XX_{\ell,E} = \linhull\{\psi_{\ell,E}\},
\end{align*}
and use the notation $\enorm{\cdot}_{-s} = \enorm{\cdot}_{H^{-s}(\Omega),\XX_L}$. Then, our first main result reads as follows:
\begin{theorem}\label{thm:main}
  The set $\XX_L$ provides a decomposition of $\PP^0(\TT_L)$, i.e.,
  \begin{align*}
    \PP^0(\TT) = \sum_{\ell=0}^L \sum_{E\in\widetilde\EE_\ell} \XX_{\ell,E}.
  \end{align*}
  Let $s\in(0,1)$. Then,
  \begin{align}
    \enorm{\phi}_{-s} \lesssim \norm{\phi}{-s} \lesssim \enorm{\phi}_{-s} \quad\text{for all } \phi\in \PP^0(\TT_L),
  \end{align}
  and the involved constants only depend on $\Omega$, $s$, $d$, the constants from~\ref{ass:mesh:reg}--\ref{ass:mesh:gen}, and $\TT_0$.
\end{theorem}
That $\XX_L$ is a decomposition of $\PP^0(\TT_L)$ follows from the proof of the lower bound given in Section~\ref{sec:proof:main:lower:uniform} (uniform meshes) resp. Section~\ref{sec:proof:main:lower:adaptive} (adaptive meshes).
The upper bound is shown in Section~\ref{sec:proof:main:upper:uniform} (uniform meshes) resp. Section~\ref{sec:proof:main:upper:adaptive} (adaptive meshes).
The proofs in the case of adaptive meshes rely on several localization arguments which are not needed in the case of uniform meshes and are therefore presented in a separate section.

\subsection{Multilevel decomposition for $\widetilde H^{-s}(\Omega)$}
We consider the collection
\begin{align*}
  \widetilde\XX_L = \{\XX_\Omega\} \cup \set{\XX_{\ell,E}}{E\in\widetilde\EE_\ell^\Omega, \, \ell=0,\dots,L},
\end{align*}
where $\XX_\Omega = \linhull\{1\}$ and set $\enorm{\cdot}_{-s,\sim} = \enorm{\cdot}_{\widetilde H^{-s}(\Omega),\widetilde\XX_L}$.
Our second main result is
\begin{theorem}\label{thm:main:tilde}
  The set $\widetilde\XX_L$ provides a decomposition of $\PP^0(\TT_L)$, i.e.,
  \begin{align*}
    \PP^0(\TT_L) = \XX_\Omega + \sum_{\ell=0}^L \sum_{E\in\widetilde\EE_\ell^\Omega} \XX_{\ell,E}.
  \end{align*}
  Let $s\in(0,1)$. Then, 
  \begin{align}
    \enorm{\phi}_{-s,\sim} \lesssim \norm{\phi}{-s,\sim} \lesssim \enorm{\phi}_{-s,\sim} \quad\text{for all } \phi\in \PP^0(\TT_L),
  \end{align}
  and the involved constants only depend on $\Omega$, $s$, $d$, the constants from~\ref{ass:mesh:reg}--\ref{ass:mesh:gen}, and $\TT_0$.
\end{theorem}
The proof in the case of uniform meshes is given in Section~\ref{sec:proof:main:tilde:lower:uniform} and Section~\ref{sec:proof:main:tilde:upper:uniform}.
The proof in the case of adaptive meshes is sketched in Section~\ref{sec:proof:main:tilde:adaptive} since it essentially follows similar steps as in the case of Theorem~\ref{thm:main}.

\subsection{Multilevel norms for $H^{-s}(\Omega)$ and $\widetilde H^{-s}(\Omega)$}
In this section we present our last two main results dealing with multilevel norms.

\begin{theorem}\label{thm:multilevelnorm}
  Let $s\in(0,1)$, then, 
  \begin{align}
    \norm{\phi}{-s}^2 \simeq \sum_{\ell=0}^L \norm{h_\ell^s (\projHone_\ell'-\projHone_{\ell-1}')\phi}{}^2 
    \quad\text{for all } \phi \in \PP^0(\TT_L).
  \end{align}
  The involved constants only depend on $\Omega$, $s$, $d$, the constants from~\ref{ass:mesh:reg}--\ref{ass:mesh:gen}, and $\TT_0$.
\end{theorem}

A similar result is valid for the space $\widetilde H^{-1}(\Omega)$:
\begin{theorem}\label{thm:multilevelnorm:tilde}
  Let $s\in(0,1)$, then, 
  \begin{align}
    \norm{\phi}{-s,\sim}^2 \simeq \sum_{\ell=0}^L \norm{h_\ell^s (\overline\projHone_\ell'-\overline\projHone_{\ell-1}')\phi}{}^2 
    \quad\text{for all } \phi \in \PP^0(\TT_L).
  \end{align}
  The involved constants only depend on $\Omega$, $s$, $d$, the constants from~\ref{ass:mesh:reg}--\ref{ass:mesh:gen}, and $\TT_0$.
\end{theorem}

A sketch of the proofs for both Theorem~\ref{thm:multilevelnorm} and Theorem~\ref{thm:multilevelnorm:tilde} is given in Section~\ref{sec:proof:main:norm}.

\begin{remark}\label{rem:extension}
  Some of our results may be extended to $s\in(-1/2,1)$ since $H^{-s}(\Omega)$ can be defined by interpolating $H^{-1}(\Omega)$ and $H^{1/2-\varepsilon}(\Omega)$ with some $\varepsilon>0$, see~\cite{CWHM15} for an overview on interpolation scales.
  In particular, the results from Section~\ref{sec:normsIntSpaces} can be extended to $s\in(-1/2,1)$ as well as our main results from Section~\ref{sec:main} for uniform meshes. 
  Major changes only involve a more general approximation property resp. inverse estimate. For a simpler presentation we consider only $s\in(0,1)$. 
\end{remark}

\section{Proof of main results}\label{sec:proof}

\subsection{Proof of lower bound in Theorem~\ref{thm:main} (uniform meshes)}\label{sec:proof:main:lower:uniform}
Let $\phi\in\PP^0(\TT_L)$ be given. 
For $\ell=0,\dots,L$ we define
\begin{align*}
  \phi_\ell := (\projHmOne_\ell-\projHmOne_{\ell-1})\phi\in \PP^0(\TT_\ell) \quad\text{with }\projHmOne_{-1}:= 0.
\end{align*}
We have that $\sum_{\ell=0}^L \phi_\ell = \projHmOne_L\phi =  \phi$.
Also recall that $\EE_\ell = \widetilde\EE_\ell$ for the sequence of uniform meshes. 
By Theorem~\ref{thm:normEquivOnelevel} and~\eqref{eq:observationLowerBound} there exists 
a stable splitting of $\phi_\ell$ into local contributions $\phi_{\ell,E}\in \XX_{\ell,E}$, i.e.,
\begin{align*}
  \phi_\ell = \sum_{E\in\EE_\ell} \phi_{\ell,E} \quad\text{and}\quad
  \sum_{E\in\EE_\ell} \norm{\phi_{\ell,E}}{-1}^2 \lesssim \norm{\phi_\ell}{-1}^2.
\end{align*}
Together with the inverse estimate $\norm{\phi_{\ell,E}}{-s} \lesssim h_\ell^{-1+s}\norm{\phi_{\ell,E}}{-1}$ we get that
\begin{align*}
  \sum_{\ell=0}^L \sum_{E\in\EE_\ell} \norm{\phi_{\ell,E}}{-s}^2 
  \lesssim \sum_{\ell=0}^L h_\ell^{-2+2s} \sum_{E\in\EE_\ell} \norm{\phi_{\ell,E}}{-1}^2
  \lesssim \sum_{\ell=0}^L h_{\ell}^{-2+2s} \norm{\phi_\ell}{-1}^2.
\end{align*}
Corollary~\ref{cor:normequiv} states that the right-hand side is equivalent to $\norm{\phi}{-s}^2$ which finishes the
proof.
\qed

\subsection{Proof of upper bound in Theorem~\ref{thm:main} (uniform meshes)}\label{sec:proof:main:upper:uniform}
Let $\phi_{\ell,E}\in \XX_{\ell,E}$ be given and 
\begin{align*}
  \phi := \sum_{\ell=0}^L \phi_\ell := \sum_{\ell=0}^L \sum_{E\in\EE_\ell} \phi_{\ell,E}.
\end{align*}
From the upper bound in Theorem~\ref{thm:normEquivOnelevel} and~\eqref{eq:observationUpperBound} we infer that
\begin{align*}
  \norm{\phi_\ell}{-1}^2 \lesssim \sum_{E\in\EE_\ell} \norm{\phi_{\ell,E}}{-1}^2.
\end{align*}
Moreover, Lemma~\ref{lem:decompositionUniform} together with the scaling $\norm{\phi_{\ell,E}}{-1}\simeq h_\ell^{1-s}\norm{\phi_{\ell,E}}{-s}$ shows that
\begin{align*}
  \norm{\phi}{-s}^2 \lesssim \sum_{\ell=0}^L h_{\ell}^{-2+2s} \norm{\phi_\ell}{-1}^2
  \lesssim \sum_{\ell=0}^L \sum_{E\in\EE_\ell} h_\ell^{-2+2s} \norm{\phi_{\ell,E}}{-1}^2 
  \simeq \sum_{\ell=0}^L \sum_{E\in\EE_\ell} \norm{\phi_{\ell,E}}{-s}^2.
\end{align*}
This concludes the proof of the upper bound in Theorem~\ref{thm:main}.
\qed

\subsection{Proof of lower bound in Theorem~\ref{thm:main:tilde} (uniform meshes)}\label{sec:proof:main:tilde:lower:uniform}
Let $\phi\in\PP^0(\TT_L)$ be given. 
For $\ell=0,\dots,L$ we define
\begin{align*}
  \phi_\ell := (\projTildeHmOne_\ell-\projTildeHmOne_{\ell-1})\phi\in \PP^0(\TT_\ell) \quad\text{with }\projTildeHmOne_{-1}:= 0.
\end{align*}
Recall that $\EE_\ell^\Omega = \widetilde\EE_\ell^\Omega$ for uniform meshes. 
We have that $\sum_{\ell=0}^L \phi_\ell = \projTildeHmOne_L \phi = \phi$.
According to Theorem~\ref{thm:normEquivOnelevel:tilde} and~\eqref{eq:observationLowerBound} we can split each $\phi_\ell$ into 
\begin{align*}
  \phi_\ell = \phi_{\ell,0} +  \sum_{E\in\EE_\ell^\Omega} \phi_{\ell,E} \quad\text{with}\quad
  \norm{\phi_{\ell,0}}{-1,\sim}^2 + \sum_{E\in\EE_\ell^\Omega} \norm{\phi_{\ell,E}}{-1,\sim}^2 \lesssim \norm{\phi_\ell}{-1,\sim}^2.
\end{align*}
We stress that $\phi_{\ell,0} = \projLtwo_\Omega \phi_\ell$ and define $\phi_0 = \sum_{\ell=0}^L \phi_{\ell,0} = \projLtwo_\Omega \phi$.
Moreover, note that $\norm{\projLtwo_\Omega\phi}{-s,\sim}\lesssim \norm{\phi}{-s,\sim}$.
Putting all estimates together and using an inverse inequality we conclude that
\begin{align*}
  \norm{\phi_0}{-s,\sim}^2 + \sum_{\ell=0}^L \sum_{E\in\EE_\ell^\Omega} \norm{\phi_{\ell,E}}{-s,\sim}^2
  &\lesssim \norm{\phi}{-s,\sim}^2 + \sum_{\ell=0}^L h_\ell^{-2+2s} \sum_{E\in\EE_\ell^\Omega} \norm{\phi_{\ell,E}}{-1,\sim}^2 \\
  &\lesssim \norm{\phi}{-s,\sim}^2 + \sum_{\ell=0}^L h_\ell^{-2+2s} \norm{\phi_\ell}{-1,\sim}^2.
\end{align*}
Applying Corollary~\ref{cor:normequiv} finishes the proof.
\qed

\subsection{Proof of upper bound in Theorem~\ref{thm:main:tilde} (uniform meshes)}\label{sec:proof:main:tilde:upper:uniform}
Let $\phi_0\in \XX_\Omega$, $\phi_{\ell,E}\in \XX_{\ell,E}$, $E\in\EE_\ell^\Omega$, $\ell=0,\dots,L$ be given and define
\begin{align*}
  \phi := \phi_0+\sum_{\ell=0}^L \phi_\ell :=  \phi_0 + \sum_{\ell=0}^L \sum_{E\in\EE_\ell^\Omega} \phi_{\ell,E}.
\end{align*}
According to Theorem~\ref{thm:normEquivOnelevel:tilde} and~\eqref{eq:observationUpperBound} we have that
\begin{align*}
  \norm{\phi_\ell}{-1,\sim}^2 \lesssim \sum_{E\in\EE_\ell^\Omega} \norm{\phi_{\ell,E}}{-1,\sim}^2.
\end{align*}
Moreover, Lemma~\ref{lem:decompositionUniform} together with the scaling $\norm{\phi_{\ell,E}}{-1,\sim}\simeq h_\ell^{1-s}\norm{\phi_{\ell,E}}{-s,\sim}$
shows that 
\begin{align*}
  \norm{\phi}{-s,\sim}^2 \lesssim \norm{\phi_0}{-s,\sim}^2 + \sum_{\ell=0}^L h_{\ell}^{-2+2s} \norm{\phi_\ell}{-1,\sim}^2
  &\lesssim \norm{\phi_0}{-s,\sim}^2 +\sum_{\ell=0}^L \sum_{E\in\EE_\ell^\Omega} h_\ell^{-2+2s} \norm{\phi_{\ell,E}}{-1,\sim}^2  \\
  &\lesssim \norm{\phi_0}{-s,\sim}^2 +\sum_{\ell=0}^L \sum_{E\in\EE_\ell^\Omega} \norm{\phi_{\ell,E}}{-s,\sim}^2.
\end{align*}
This concludes the proof of the upper bound in Theorem~\ref{thm:main:tilde}.
\qed

\subsection{Proof of lower bound in Theorem~\ref{thm:main} (adaptive meshes)}\label{sec:proof:main:lower:adaptive}
Let $\phi \in \PP^0(\TT_L)$ and consider the splitting
\begin{align}\label{eq:splittingPhi}
  \sum_{\ell=0}^L \phi_\ell := \sum_{\ell=0}^L (\projHmOneSZ_\ell-\projHmOneSZ_{\ell-1})\phi = \projHmOneSZ_L \phi = \phi.
\end{align}
Throughout we make the convention that operators with negative indices are trivial, i.e., $\projLtwo_{k} := 0$, $\projHone_k':=0$, $\projHmOneSZ_k:=0$ for all $k<0$. 

The next result provides a decomposition of $\phi=\sum_{\ell=0}^L\phi_\ell$ and a stability result:
\begin{lemma}\label{lem:stability:splittingPhiEll}
  There exists a decomposition $\phi = \sum_{\ell=0}^L \sum_{E\in\widetilde\EE_\ell} \phi_{\ell,E}$ with $\phi_{\ell,E}\in \XX_{\ell,E}$, $E\in\widetilde\EE_\ell$, $\ell\in\N_0$ such that
  \begin{align*}
    \sum_{\ell=0}^L\sum_{E\in\widetilde\EE_\ell} \norm{\phi_{\ell,E}}{-s}^2 \lesssim \norm{\projHmOneSZ_0\phi}{-1}^2 + 
    \sum_{\ell=1}^L \Big( \norm{h_\ell^s (\projLtwo_\ell-\projLtwo_{\ell-1})\projHone_\ell'\phi}{}^2
    + \norm{h_\ell^s(\projHone_\ell'-\projHone_{\ell-1}')\phi}{}^2\Big).
  \end{align*}
  The involved constant only depends on $\Omega$, $s$, $d$, the constants from~\ref{ass:mesh:reg}--\ref{ass:mesh:ref}, and $\TT_0$.
\end{lemma}
\begin{proof}
\textbf{Step 1.}
According to Theorem~\ref{thm:normEquivOnelevel} there exist functions $\phi_{0,E}\in\XX_{0,E}$, $E\in\EE_0=\widetilde\EE_0$ such that
\begin{align*}
  \phi_0 = \projHmOneSZ_0\phi =  \sum_{E\in\widetilde\EE_0} \phi_{0,E} \quad\text{and}\quad
  \sum_{E\in\widetilde\EE_0} \norm{\phi_{0,E}}{-1}^2 \lesssim \norm{\phi_0}{-1}^2 = \norm{\projHmOneSZ_0\phi}{-1}^2.
\end{align*}
Using the inverse estimate and $h_E \simeq h_0 \simeq 1$ for all $E\in\widetilde\EE_0$ we get that
\begin{align*}
  \sum_{E\in\widetilde\EE_0} \norm{\phi_{0,E}}{-s}^2 \lesssim \sum_{E\in\widetilde\EE_0} h_E^{-2+2s} \norm{\phi_{0,E}}{-1}^2
  \lesssim \norm{\projHmOneSZ_0\phi}{-1}^2.
\end{align*}

Let $\ell\geq1$.
We split $\phi_\ell$ into two contributions $\phi_{\ell,1}$ and $\phi_{\ell,2}$, 
\begin{align}\label{eq:splittingPhi1Phi2}
  \phi_\ell = \phi_{\ell,1} + \phi_{\ell,2} := 
  \left(\projLtwo_\ell\projHone_\ell'-\projLtwo_{\ell-1}\projHone_\ell'\right)\phi 
  +\left(\projLtwo_{\ell-1}\projHone_\ell'-\projLtwo_{\ell-1}\projHone_{\ell-1}'\right)\phi.
\end{align}
The first term can be localized on each $T\in\TT_{\ell-1}\setminus\TT_\ell$ whereas the second term can be localized using the partition of unity provided by the nodal functions $\set{\eta_{\ell-1,z}}{z\in\NN_{\ell-1}}$.

\noindent
\textbf{Step 2.}
First, we provide a decomposition for $\phi_{\ell,1}$.
To that end let $T\in \TT_{\ell-1}\setminus \TT_\ell$ and consider the local Neumann problem
\begin{subequations}\label{eq:localProblemPhi1}
\begin{alignat}{2}
  \Delta u_T &= \phi_{\ell,1} &\quad& \text{in } T, \\
  \partial_{\normal} u_T &= 0 &\quad& \text{on } \partial T \label{eq:localProblemPhi1:BC}.
\end{alignat}
\end{subequations}
This problem admits a unique solution $u_T\in H_*^1(T) = \set{v\in H^1(T)}{\ip{v}1_T=0}$ since
\begin{align*}
  \ip{\phi_{\ell,1}}1_T &= \ip{(\projLtwo_\ell\projHone_\ell'-\projLtwo_{\ell-1}\projHone_\ell')\phi}1_T
  = \ip{\projHone_\ell'\phi}{(\projLtwo_\ell-\projLtwo_{\ell-1})1}_T = 0.
\end{align*}
From the weak formulation and a Poincar\'e inequality ($\norm{v}{T} \lesssim h_T \norm{\nabla v}T$ for $v\in H_*^1(T)$) we deduce the stability estimate
\begin{align}\label{eq:stability:uT}
  \norm{\nabla u_T}{T} \lesssim h_T \norm{\phi_{\ell,1}}T.
\end{align}
In a further step we use the operator from Lemma~\ref{lem:Hdivproj} and define
\begin{align*}
  \ssigma_T := \begin{cases} 
    \projHdiv_{\TT_\ell(T)}^0\nabla u_T & \text{on } T, \\
    0 & \text{on } \Omega\setminus T.
  \end{cases}
\end{align*}
Here, $\TT_\ell(T) = \set{T'\in\TT_\ell}{T'\subset T}$. Note that by definition of the operator the normal trace of $\ssigma_T$ is zero on $\partial T$ and thus $\ssigma_T \in \Hdivset\Omega$. Furthermore,
\begin{align*}
  \ssigma_T \in \linhull\set{\ppsi_{\ell,E}}{E\in \widetilde\EE_\ell, \, \supp{\psi_{\ell,E}} \subset T}
\end{align*}
and the commutativity property gives
\begin{align*}
  \div\ssigma_T|_T = \div \projHdiv_{\TT_\ell(T)}^0 \nabla u_T = \projLtwo_{\TT_\ell(T)} \phi_{\ell,1} = \phi_{\ell,1}|_T.
\end{align*}
Setting $\ssigma_{\ell,1} := \sum_{T\in\TT_{\ell-1}\setminus\TT_{\ell}} \ssigma_T$ we conclude with the aforementioned properties that there exist coefficients $\alpha_{\ell,E,1}$ such that
\begin{align*}
  \ssigma_{\ell,1} = \sum_{E\in\widetilde\EE_\ell} \alpha_{\ell,E,1} \ppsi_{\ell,E} 
\end{align*}
and
\begin{align*}
  \sum_{E\in\widetilde\EE_\ell} \phi_{\ell,E,1} := \sum_{E\in\widetilde\EE_\ell} \alpha_{\ell,E,1} \div\ppsi_{\ell,E} 
  = \div\ssigma_{\ell,1} = \phi_{\ell,1}.
\end{align*}
Moreover, Lemma~\ref{lem:Hdivproj} and~\eqref{eq:stability:uT} also imply that 
\begin{align*}
  \norm{\projLtwo_{\TT_\ell(T)}\ssigma_T}{T} \lesssim \norm{\ssigma_T}T \lesssim h_T\norm{\phi_{\ell,1}}T.
\end{align*}
The scaling properties $\norm{\phi_{\ell,E,1}}{-s}\simeq h_E^{-1+s}\norm{\phi_{\ell,E,1}}{-1}\simeq  h_E^{-1+s}\norm{\alpha_{\ell,E,1}\ppsi_{\ell,E}}{}$ (Lemma~\ref{lem:scaling}), the $L^2$ stability of the Raviart--Thomas basis (\cite[Proposition~2]{FuehrerHeuerQuasiDiagonal19}) show that
\begin{align*}
  \sum_{E\in\widetilde\EE_\ell} \norm{\phi_{\ell,E,1}}{-s}^2 &\simeq \sum_{E\in\widetilde\EE_\ell} h_E^{-2+2s}\norm{\alpha_{\ell,E,1}\ppsi_{\ell,E}}{}^2
  \simeq \norm{h_\ell^{-1+s}\ssigma_{\ell,1}}{}^2 \lesssim \norm{h_\ell^{s}\phi_{\ell,1}}{}^2.
\end{align*}

\noindent
\textbf{Step 3.}
We define local problems for the second contribution $\phi_{\ell,2}$ in~\eqref{eq:splittingPhi1Phi2}:
To that end we use the partition of unity $1=\sum_{z\in\NN_{\ell-1}}\eta_{\ell-1,z}$ and consider for $z\in\NN_{\ell-1}$ the problem
\begin{subequations}\label{eq:localProblemPhi2}
\begin{alignat}{2}
  \Delta u_z &= \projLtwo_{\ell-1}\eta_{\ell-1,z}(\projHone_\ell'-\projHone_{\ell-1}')\phi &\quad& \text{in } \Patch_{\ell-1}(z), \\
  \partial_{\normal} u_z &= 0 &\quad& \text{on } \partial \Patch_{\ell-1}(z)\setminus \Gamma, \label{eq:localProblemPhi2:BC1} \\
  u_z &= 0 &\quad& \text{on } \partial \Patch_{\ell-1}(z)\cap \Gamma \text{ if } z \in \Gamma. 
\label{eq:localProblemPhi2:BC2}
\end{alignat}
\end{subequations}
For an interior node $z\in \NN_{\ell-1}^\Omega$ we have by Lemma~\ref{lem:projHone} that $(\projHone_\ell-\projHone_{\ell-1})\eta_{\ell-1,z} = 0$ and therefore
\begin{align*}
  \ip{\projLtwo_{\ell-1}\eta_z(\projHone_\ell'-\projHone_{\ell-1}')\phi}1_{\Patch_{\ell-1}(z)}
  &= \ip{(\projHone_\ell'-\projHone_{\ell-1}')\phi}{\eta_{\ell-1,z}}_{\Patch_{\ell-1}(z)} \\
  &= \ip{(\projHone_\ell'-\projHone_{\ell-1}')\phi}{\eta_{\ell-1,z}} \\
  &= \ip{\phi}{(\projHone_\ell-\projHone_{\ell-1})\eta_{\ell-1,z}} = 0.
\end{align*}
Thus, there exists a unique solution $u_z\in H_*^1(\Patch_{\ell-1}(z))$ for all $z\in\NN_{\ell-1}^\Omega$.
For $z\in\NN_{\ell-1}^\Gamma$ the surface measure $|\partial \Patch_{\ell-1}(z)\cap \Gamma|$ is positive (at least one boundary facet is contained in that set).
Thus, there exists a unique solution $u_z \in H_{\Gamma}^1(\Patch_{\ell-1}(z)) := \set{v\in H^1(\Patch_{\ell-1}(z))}{v|_\Gamma = 0}$ if $z\in\NN_{\ell-1}^\Gamma$. 
The weak formulation of~\eqref{eq:localProblemPhi2} and Poincar\'e--Friedrichs' inequalities lead to
\begin{align*}
  \norm{v}{\Patch_{\ell-1}(z)} \lesssim \diam(\Patch_{\ell-1}(z)) \norm{\nabla v}{\Patch_{\ell-1}(z)} 
  \quad\text{for all } v\in H_\Gamma^1(\Patch_{\ell-1}(z)) \text{ with } z\in \NN_{\ell-1}^\Gamma.
\end{align*}
Consequently, the weak formulation of~\eqref{eq:localProblemPhi2}, the latter estimate, local boundedness of $\projLtwo_{\ell-1}$ and $\norm{\eta_{\ell-1,z}}\infty \leq 1$ show that
\begin{align*}
  \norm{\nabla u_z}{\Patch_{\ell-1}(z)} &\lesssim \diam(\Patch_{\ell-1}(z)) 
  \norm{\projLtwo_{\ell-1}\eta_{\ell-1,z}(\projHone_\ell'-\projHone_{\ell-1}')\phi}{\Patch_{\ell-1}(z)}
  \\
  &\leq \diam(\Patch_{\ell-1}(z)) \norm{(\projHone_\ell'-\projHone_{\ell-1}')\phi}{\Patch_{\ell-1}(z)}
\end{align*}
for all $z\in\NN_{\ell-1}$. 
As before we use the operator from Lemma~\ref{lem:Hdivproj} and define
\begin{align*}
  \ssigma_z := \begin{cases} 
    \projHdiv_{\patch_{\ell-1}(z)}^0\nabla u_z & \text{on } \Patch_{\ell-1}(z), \\
    0 & \text{on } \Omega\setminus \Patch_{\ell-1}(z).
  \end{cases}
\end{align*}
Note that by definition of the operator the normal trace of $\ssigma_z$ is zero on all facets $E\in \EE_\ell$ with $E\subset \partial\Patch_{\ell-1}(z)\setminus\Gamma$ 
and thus $\ssigma_z \in \Hdivset\Omega$. 
Furthermore,
\begin{align*}
  \ssigma_z \in \linhull\set{\ppsi_{\ell,E}}{E\in \EE_\ell, \, \supp{\psi_{\ell,E}} \subset \Patch_{\ell-1}(z)}.
\end{align*}
Set $\NN_{\ell-1}' := \set{z\in\NN_{\ell-1}}{\ssigma_z|_{\Patch_{\ell-1}(z)} = 0}$, $\overline\NN_{\ell-1} := \NN_{\ell-1}\setminus \NN_{\ell-1}'$ and
\begin{align*}
  \overline\EE_\ell := \set{E\in\EE_\ell}{\supp(\psi_{\ell,E})\subset \Patch_{\ell-1}(z) \text{ for some } z\in\overline\NN_{\ell-1}}.
\end{align*}

Define $\ssigma_{\ell,2} := \sum_{z\in\overline\NN_{\ell-1}} \ssigma_z =: \sum_{E\in\overline\EE_\ell} \beta_{\ell,E} \ppsi_{\ell,E}$.
The same arguments as in Step~2 then show that $\div\ssigma_{\ell,2} = \phi_{\ell,2}$ and 
\begin{align*}
  \sum_{E\in\overline\EE_\ell} \norm{\beta_{\ell,E}\psi_{\ell,E}}{-s}^2 \lesssim \norm{h_\ell^s(\projHone_\ell'-\projHone_{\ell-1}')\phi}{}^2. 
\end{align*}
Note that in general $\widetilde\EE_\ell\neq \overline\EE_\ell$. To tackle this issue we stress that there exist $\alpha_{\ell,2,E}$ such that
\begin{align*}
  \sum_{\ell=0}^L \sum_{E\in\widetilde\EE_\ell} \phi_{\ell,2,E} := \sum_{\ell=0}^L \sum_{E\in\widetilde\EE_\ell} \alpha_{\ell,2,E} \psi_{\ell,E}
  = \sum_{\ell=1}^L \sum_{E\in\overline\EE_\ell} \beta_{\ell,E} \psi_{\ell,E}.
\end{align*}
To see this note that if $E\in\overline\EE_\ell\setminus\widetilde\EE_\ell$ then there exists $k\leq \ell-1$, $E\in\widetilde\EE_k$ and $\psi_{\ell,E} = \psi_{k,E}$. 
Furthermore, given $E\in\widetilde\EE_\ell$ we have that $\#\set{k\in\N_0}{E\in\overline\EE_k\setminus\widetilde\EE_k}\lesssim 1$. This means that the number of coefficients $\beta_{k,E}$ that contribute to $\alpha_{\ell,2,E}$ is uniformly bounded.
Thus, we have shown that there exist $\phi_{\ell,2,E}\in \XX_{\ell,E}$ with
\begin{align*}
  \sum_{\ell=0}^L \sum_{E\in\widetilde\EE_\ell} \norm{\phi_{\ell,2,E}}{-s}^2 \lesssim \sum_{\ell=1}^L \sum_{E\in\overline\EE_\ell} \norm{\beta_{\ell,E}\psi_{\ell,E}}{-s}^2 
  \lesssim \sum_{\ell=1}^L \norm{h_\ell^s(\projHone_\ell'-\projHone_{\ell-1}')\phi}{}^2.
\end{align*}

\noindent
\textbf{Step 4.}
Combining the results from Step~1--3 above we conclude that there exist $\phi_{\ell,E}\in \XX_{\ell,E}$ for all $E\in\widetilde\EE_\ell$ and $\ell=0,\dots,L$ such that
\begin{align*}
  \phi = \sum_{\ell=0}^L \sum_{E\in\widetilde\EE_\ell} \phi_{\ell,E},
\end{align*}
and
\begin{align*}
  \sum_{\ell=0}^L \sum_{E\in\widetilde\EE_\ell} \norm{\phi_{\ell,E}}{-s}^2
  \lesssim \norm{\phi_0}{-1}^2 + \sum_{\ell=1}^L \Big(\norm{h_\ell^s \phi_{\ell,1}}{}^2 + \norm{h_\ell^s(\projHone_\ell'-\projHone_{\ell-1}')\phi}{}^2\Big),
\end{align*}
which finishes the proof.
\end{proof}

For the remainder of the proof we group the contributions $\phi_{\ell,E}$ that have the same scale (i.e., the supports are comparable).
This allows us to establish a connection to a sequence of uniform refined meshes and, consequently, we can apply Theorem~\ref{thm:normequivInterpolation}. 
The technique of relating uniform and adaptive meshes is quite common and used in, e.g.,~\cite{CHX13,CNX12,CXZ12,HiptmairWuZheng2012,HypSing17}.
To that end we need some further notation: 
The collection of elements in the sequence $(\TT_\ell)_{\ell=0}^L$ is denoted by 
\begin{align*}
  \TT_\mathrm{tot} = \TT_0 \cup \bigcup_{\ell=1}^L \TT_\ell\setminus\TT_{\ell-1}.
\end{align*}

Let $(\widehat\TT_m)_{m\in\N_0}$ denote a sequence of uniform refined meshes with $\widehat\TT_0 = \TT_0$.
All quantities related to $\widehat\TT_m$ will be denoted with $\widehat h_m$, $\widehat\patch_m$, $\projHmOneHat_m$, etc. 
We set
\begin{align*}
  \widehat\TT_\mathrm{tot} = \bigcup_{m\in\N_0} \widehat\TT_m.
\end{align*}

The next result follows from our assumptions on the mesh refinement given in Section~\ref{sec:mesh}.
\begin{lemma}\label{lem:meshprop}
  There exists $\mfun\colon \TT_\mathrm{tot}\to \N_0$, $\tfun \colon \TT_\mathrm{tot}\to \widehat\TT_\mathrm{tot}$ and $k\in\N$ such that
  \begin{enumerate}[label=(\alph*)]
    \item\label{lem:meshprop:a} $h_T \simeq \widehat h_{\mfun(T)}$ for all $T\in\TT_\mathrm{tot}$,
    \item\label{lem:meshprop:d} $\Patch_{\ell-1}^{(2)}(T)\subset \widehat\Patch_{\mfun(T)}^{(k)}(\tfun(T))$ for all $T\in\TT_\ell\setminus\TT_{\ell-1}$ and $\ell\geq 1$,
    \item\label{lem:meshprop:de} $\psi\in\PP^0(\widehat\patch_{\mfun(T)}^{(k)}(\tfun(T)))$ implies that $\psi|_{\Patch_{\ell-1}^{(2)}(T)}\in\PP^0(\patch_{\ell-1}^{(2)}(T))$ for all $T\in\TT_\ell\setminus\TT_{\ell-1}$ and $\ell\geq 1$,
    \item\label{lem:meshprop:e} $\#\set{T\in\TT_\mathrm{tot}}{\mfun(T)=m \text{ and } \tfun(T) = \widehat T} \lesssim 1$ for all $\widehat T \in \widehat\TT_m$ and all $m\in\N_0$.
  \end{enumerate}
  The involved constant only depends on the constants from~\ref{ass:mesh:reg}--\ref{ass:mesh:gen}.
\end{lemma}

We recall the following summation property, see, e.g.~\cite[Theorem~4.1]{amt99}.
\begin{lemma}\label{lem:sumHmOne}
  For two non-empty disjoint $\omega_1,\omega_2\subset \Omega$ with $\omega = \intr(\overline\omega_1\cup\overline\omega_2)$ and $\psi\in L^2(\omega)$,
  \begin{align*}
    \norm{\psi|_{\omega_1}}{-1,\omega_1}^2 + \norm{\psi|_{\omega_2}}{-1,\omega_2}^2 \leq \norm{\psi}{-1,\omega}^2.
  \end{align*}
\end{lemma}

To complete the proof of the lower bound in Theorem~\ref{thm:main} it remains to show that
\begin{align*}
  \sum_{\ell=1}^L \Big(\norm{h_\ell^s(\Pi_\ell^0-\Pi_{\ell-1}^0)\projHone_\ell'\phi}{}^2 + \norm{h_\ell^s(\projHone_\ell'-\projHone_{\ell-1}')\phi}{}^2\Big) \lesssim \norm{\phi}{-s}^2.
\end{align*}

Let $T\in\TT_\ell\setminus\TT_{\ell-1}$. Recall the definitions of $\mfun$, $\tfun$ and $k$ from Lemma~\ref{lem:meshprop}.
Note that $\Patch_{\ell-1}(T)$ contains the father element of $T$. Furthermore, $\projHone_\ell'$ is a local quasi-projection which
satisfies that $\projHone_\ell' \psi|_{\Patch_{\ell-1}(T)} = \psi|_{\Patch_{\ell-1}(T)}$ for $\psi|_{\Patch^{(2)}_{\ell-1}(T)}\in\PP^0(\patch_{\ell-1}^{(2)}(T))$.
Combining these arguments together with an inverse estimate, local boundedness of $\projHone_\ell'$ (Lemma~\ref{lem:projHoneDual}), Lemma~\ref{lem:meshprop} and Lemma~\ref{lem:sumHmOne}, proves that
\begin{align*}
  h_T^s \norm{(\Pi_\ell^0-\Pi_{\ell-1}^0)\projHone_\ell'\phi}T
  &= h_T^s \norm{(\Pi_\ell^0-\Pi_{\ell-1}^0)\projHone_\ell'(\phi-\projHmOneHat_{\mfun(T)}\phi) }T
  \\
  &\leq h_T^s \norm{(\Pi_\ell^0-\Pi_{\ell-1}^0)\projHone_\ell'(\phi-\projHmOneHat_{\mfun(T)}\phi) }{\Patch_{\ell-1}(T)}
  \lesssim h_T^s\norm{\projHone_\ell'(\phi-\projHmOneHat_{\mfun(T)}\phi)}{\Patch_{\ell-1}(T)}
  \\
  & \lesssim h_T^{-1+s} \Big( \sum_{T'\in\TT_\ell, T'\subset \Patch_{\ell-1}(T)}  \norm{\projHone_\ell'(\phi-\projHmOneHat_{\mfun(T)}\phi)}{-1,T'}^2\Big)^{1/2}
  \\
  & \lesssim h_T^{-1+s} \Big( \sum_{T'\in\TT_\ell, T'\subset \Patch_{\ell-1}(T)}  \norm{\phi-\projHmOneHat_{\mfun(T)}\phi}{-1,\Omega_\ell(T')}^2\Big)^{1/2}
  \\
  &\lesssim h_T^{-1+s} \norm{\phi-\projHmOneHat_{\mfun(T)}\phi}{-1,\Patch_{\ell-1}^{(2)}(T)} 
  \lesssim \widehat h_{\mfun(T)}^{-1+s} \norm{\phi-\projHmOneHat_{\mfun(T)}\phi}{-1,\widehat\Patch_{\mfun(T)}^{(k)}(\tfun(T))}.
\end{align*}
Therefore, the latter estimate and Lemma~\ref{lem:meshprop}~\ref{lem:meshprop:e} yield
\begin{align*}
  \sum_{\ell=1}^L \sum_{T\in\TT_\ell\setminus\TT_{\ell-1}} \norm{h_\ell^s (\Pi_\ell^0-\Pi_{\ell-1}^0)\projHone_\ell'\phi}T^2 
  &\lesssim \sum_{T\in\TT_\mathrm{tot}} \widehat h_{\mfun(T)}^{-2+2s} \norm{\phi-\projHmOneHat_{\mfun(T)}\phi}{-1,\widehat\Patch_{\mfun(T)}^{(k)}(\tfun(T))}^2
  \\
  &= \sum_{m=0}^\infty \sum_{\widehat T\in\widehat\TT_m} \sum_{\substack{T\in\TT_\mathrm{tot}\\ \mfun(T)=m, \, \tfun(T)=\widehat T}}
  \widehat h_{\mfun(T)}^{-2+2s} \norm{\phi-\projHmOneHat_{\mfun(T)}\phi}{-1,\widehat\Patch_{\mfun(T)}^{(k)}(\tfun(T))}^2
  \\
  &\lesssim \sum_{m=0}^\infty \sum_{\widehat T\in\widehat\TT_m} 
  \widehat h_{m}^{-2+2s} \norm{\phi-\projHmOneHat_m\phi}{-1,\widehat\Patch_{m}^{(k)}(\widehat T)}^2.
\end{align*}
A standard coloring argument and Lemma~\ref{lem:sumHmOne} show that
\begin{align*}
\sum_{\widehat T\in\widehat\TT_m} 
  \widehat h_{m}^{-2+2s} \norm{\phi-\projHmOneHat_m\phi}{-1,\widehat\Patch_{m}^{(k)}(\widehat T)}^2
  \lesssim \widehat h_m^{-2+2s} \norm{\phi-\projHmOneHat_m\phi}{-1}^2.
\end{align*}
Then, with the norm equivalence from Theorem~\ref{thm:normequivInterpolation} we conclude that
\begin{align*}
  \sum_{\ell=1}^L \norm{h_\ell^s(\Pi_\ell^0-\Pi_{\ell-1}^0)\projHone_\ell'\phi}{}^2
  \lesssim \sum_{m=0}^\infty \widehat h_m^{-2+2s} \norm{\phi-\projHmOneHat_m\phi}{-1}^2 \lesssim \norm{\phi}{-s}^2.
\end{align*}
It remains to prove that
\begin{align*}
  \sum_{\ell=1}^L \norm{h_\ell^s(\projHone_\ell'-\projHone_{\ell-1}')\phi}{}^2 \lesssim \norm{\phi}{-s}^2.
\end{align*}
This estimate follows with similar arguments as given above since $\supp(\projHone_\ell'-\projHone_{\ell-1}')\phi$ is contained in a patch of fixed order around $\TT_\ell\setminus\TT_{\ell-1}$ and $\projHone_\ell'$ resp. $\projHone_{\ell-1}'$ restricted to piecewise constants are local projections.
Thus, for the decomposition from Lemma~\ref{lem:stability:splittingPhiEll} we have proven that
\begin{align*}
  \phi = \sum_{\ell=0}^L \sum_{E\in\widetilde\EE_\ell} \phi_{\ell,E} \quad\text{and}\quad
  \sum_{\ell=0}^L \sum_{E\in\widetilde\EE_\ell} \norm{\phi_{\ell,E}}{-s}^2 \lesssim \norm{\phi_0}{-1}^2 +  \norm{\phi}{-s}^2
  \lesssim \norm{\phi}{-s}^2,
\end{align*}
where in the last estimate we used $\norm{\phi_0}{-1}\lesssim \norm{\phi}{-1}\lesssim \norm{\phi}{-s}$. 
This together with~\eqref{eq:observationLowerBound} finishes the proof of the lower bound in Theorem~\ref{thm:main} for the case of adaptive meshes.
\qed

\subsection{Proof of upper bound in Theorem~\ref{thm:main} (adaptive meshes)}\label{sec:proof:main:upper:adaptive}
Let $\phi_{\ell,E}\in \XX_{\ell,E}$, $E\in\widetilde\EE_\ell$, $\ell=0,\dots,L$ be given and set
\begin{align*}
  \phi := \sum_{\ell=0}^L \sum_{E\in\widetilde\EE_\ell} \phi_{\ell,E}.
\end{align*}
Recall that by~\eqref{eq:observationUpperBound} the proof of the upper bound is equivalent to show that
\begin{align*}
  \norm{\phi}{-s}^2 \lesssim \sum_{\ell=0}^L \sum_{E\in\widetilde\EE_\ell} \norm{\phi_{\ell,E}}{-s}^2.
\end{align*}

The basic idea is to reorder the sum into contributions of the same scale, similar as in Section~\ref{sec:proof:main:lower:adaptive}. 
Let $(\widehat\TT_m)_{m\in\N_0}$ denote the sequence of uniform meshes with $\widehat\TT_0=\TT_0$.
The next result follows from properties of the mesh refinement given in Section~\ref{sec:mesh}.
\begin{lemma}\label{lem:meshpropUp}
  Consider $\II := \set{(\ell,E)}{E\in\widetilde\EE_\ell, \, \ell=0,\dots,L}$.
  There exists $\mfunUp\colon \II\to \N_0$ such that
  \begin{enumerate}[label=(\alph*)]
    \item $h_E \simeq \widehat h_{\mfunUp(\ell,E)}$ for all $(\ell,E)\in\II$,
    \item $\ppsi_{\ell,E} \in \RT^0(\widehat\TT_{\mfunUp(\ell,E)})$ and $\phi_{\ell,E} \in \PP^0(\widehat\TT_{\mfunUp(\ell,E)})$ for all $(\ell,E)\in\II$.
    \item $\#\set{(\ell,E)\in\II}{\mfunUp(\ell,E)=m \text{ and } \supp(\phi_{\ell,E})\cap T \neq \emptyset}\lesssim 1$ for all $T\in\widehat T_m$ and $m\in\N_0$.
  \end{enumerate}
  The involved constants depend only on the constants from~\ref{ass:mesh:reg}--\ref{ass:mesh:gen}.
\end{lemma}

We rewrite $\phi$ as
\begin{align*}
  \phi = \sum_{\ell=0}^L \sum_{E\in\widetilde\EE_\ell} \phi_{\ell,E} = \sum_{m=0}^\infty \sum_{\ell=0}^L \sum_{\substack{E\in\widetilde\EE_\ell\\ \mfunUp(\ell,E)=m}} 
  \phi_{\ell,E} =: \sum_{m=0}^\infty \widehat\phi_m.
\end{align*}
Note that there exists $M\in\N$ with $\widehat\phi_m=0$ for all $m>M$. 
Since $\widehat\phi_m \in \PP^0(\widehat\TT_m)$ we use Lemma~\ref{lem:decompositionUniform} to see that
\begin{align*}
  \norm{\phi}{-s}^2 \lesssim \sum_{m=0}^M \widehat h_m^{-2+2s} \norm{\widehat\phi_m}{-1}^2.
\end{align*}
Then, observe that 
\begin{align*}
  \widehat\phi_m = \sum_{\ell=0}^L \sum_{\substack{E\in\widetilde\EE_\ell \\ \mfunUp(\ell,E)=m}} \phi_{\ell,E} 
  =: \sum_{\ell=0}^L \sum_{\substack{E\in\widetilde\EE_\ell \\ \mfunUp(\ell,E)=m}} \div(\alpha_{\ell,E}\ppsi_{\ell,E})
  = \div\Big( \sum_{\ell=0}^L \sum_{\substack{E\in\widetilde\EE_\ell \\ \mfunUp(\ell,E)=m}} \alpha_{\ell,E}\ppsi_{\ell,E} \Big) =: \div\ppsi_m.
\end{align*}
Using that $\div\colon L^2(\Omega)^d\to H^{-1}(\Omega)$ is a bounded operator we infer that
\begin{align*}
  \widehat h_m^{-2+2s} \norm{\widehat\phi_m}{-1}^2 \leq \widehat h_m^{-2+2s} \norm{\ppsi_m}{}^2 = 
  \sum_{T\in\widehat \TT_m} \widehat h_m^{-2+2s} \norm{\ppsi_m}{T}^2 
\end{align*}
From Lemma~\ref{lem:meshpropUp} we obtain that $\widehat h_m\simeq h_T \simeq h_E$ for all $(\ell,E)$ with $\mfunUp(\ell,E)=m$ and $\supp\phi_{\ell,E} \cap T\neq \emptyset$ as well as
that the number of functions $\phi_{\ell,E}$ with $\mfunUp(\ell,E)=m$ that do not vanish on $T\in\widehat\TT_m$ is uniformly bounded by a constant leading to
\begin{align*}
  \sum_{T\in\widehat \TT_m} \widehat h_m^{-2+2s} \norm{\ppsi_m}{T}^2 \lesssim \sum_{T\in\widehat\TT_m} \sum_{\ell=0}^L \sum_{\substack{E\in\widetilde\EE_\ell \\ \mfunUp(\ell,E)=m}} h_E^{-2+2s} \norm{\alpha_{\ell,E}\ppsi_{\ell,E}}{T}^2 = \sum_{\ell=0}^L \sum_{\substack{E\in\widetilde\EE_\ell \\ \mfunUp(\ell,E)=m}} h_E^{-2+2s} \norm{\alpha_{\ell,E}\ppsi_{\ell,E}}{}^2.
\end{align*}
Recall the scaling $h_E^{-1+s}\norm{\alpha_{\ell,E}\ppsi_{\ell,E}}{} \simeq \norm{\phi_{\ell,E}}{-s}$.
Combining all estimates above we conclude that
\begin{align*}
  \norm{\phi}{-s}^2 &\lesssim \sum_{m=0}^M \widehat h_m^{-2+2s} \norm{\widehat\phi_m}{-1}^2 
  \lesssim \sum_{m=0}^M \sum_{\ell=0}^L \sum_{\substack{E\in\widetilde\EE_\ell \\ \mfunUp(\ell,E)=m}} h_E^{-2+2s} \norm{\alpha_{\ell,E}\ppsi_{\ell,E}}{}^2 
  \\
  &\simeq \sum_{m=0}^M \sum_{\ell=0}^L \sum_{\substack{E\in\widetilde\EE_\ell \\ \mfunUp(\ell,E)=m}} \norm{\phi_{\ell,E}}{-s}^2
  = \sum_{\ell=0}^L \sum_{E\in\widetilde\EE_\ell} \norm{\phi_{\ell,E}}{-s}^2,
\end{align*}
which finishes the proof. \qed

\subsection{Proof of Theorem~\ref{thm:main:tilde} (adaptive meshes)}\label{sec:proof:main:tilde:adaptive}
We only give a sketch of the proof since most arguments are the same as in Section~\ref{sec:proof:main:lower:adaptive} and Section~\ref{sec:proof:main:upper:adaptive}.

For the proof of the upper bound let $\phi_0\in\XX_\Omega$, $\phi_{\ell,E}\in\XX_{\ell,E}$, $E\in\widetilde\EE_\ell^\Omega$, $\ell=0,\dots,L$ be given and 
\begin{align*}
  \phi:= \phi_0+\phi_* := \phi_0 + \sum_{\ell=0}^L \sum_{E\in\widetilde\EE_\ell^\Omega} \phi_{\ell,E}.
\end{align*}
Therefore, 
  $\norm{\phi}{-s,\sim}^2 \lesssim \norm{\phi_0}{-s,\sim}^2 + \norm{\phi_*}{-s,\sim}^2$.
The same arguments as in Section~\ref{sec:proof:main:upper:adaptive} show that
\begin{align*}
  \norm{\phi_*}{-s,\sim}^2 \lesssim \sum_{\ell=0}^L \sum_{E\in\widetilde\EE_\ell^\Omega} \norm{\phi_{\ell,E}}{-s,\sim}^2.
\end{align*}
Putting all estimates together this proves the upper bound in Theorem~\ref{thm:main:tilde}.

To see the lower bound let $\phi\in\PP^0(\TT_L)$ be given and set $\phi_0 := \projLtwo_\Omega\phi \in \XX_0$, $\phi_*:=\phi-\phi_0$. Then,
\begin{align*}
  \norm{\phi_0}{-s,\sim}^2 + \norm{\phi_*}{-s,\sim}^2 \lesssim \norm{\phi}{-s,\sim}^2.
\end{align*}
Following the arguments from Lemma~\ref{lem:stability:splittingPhiEll} we deduce that there exist $\phi_{\ell,E}\in\XX_{\ell,E}$  with
\begin{align*}
  \phi_* = \sum_{\ell=0}^L \sum_{E\in\widetilde\EE_\ell^\Omega} \phi_{\ell,E}
\end{align*}
and that
\begin{align}\label{eq:main:tilde:adap}
  \sum_{\ell=0}^L \sum_{E\in\widetilde\EE_\ell^\Omega} \norm{\phi_{\ell,E}}{-s,\sim}^2 \lesssim \norm{\phi_*}{-s,\sim}^2
  + \sum_{\ell=1}^L \Big(\norm{h_\ell^s (\projLtwo_\ell-\projLtwo_{\ell-1})\overline\projHone_\ell'\phi_*}{}^2 
  + \norm{h_\ell^s(\overline\projHone_\ell'-\overline\projHone_{\ell-1}')\phi_*}{}^2\Big).
\end{align}
The major differences are that instead of $\projHone_\ell'$, $\projHmOneSZ_\ell$ we use $\overline\projHone_\ell'$, $\widetilde\projHmOneSZ_\ell$ and instead of the local problem~\eqref{eq:localProblemPhi2} we consider
\begin{alignat*}{2}
  \Delta u_z &= \projLtwo_{\ell-1}\eta_{\ell-1,z}(\overline\projHone_\ell'-\overline\projHone_{\ell-1}')\phi &\quad& \text{in } \Patch_{\ell-1}(z), \\
  \partial_{\normal} u_z &= 0 &\quad& \text{on } \partial \Patch_{\ell-1}(z), 
\end{alignat*}
i.e., pure Neumann boundary conditions. 
Moreover, we replace the definition of $\HdivsetZero{\widetilde\Patch}$ in Lemma~\ref{lem:Hdivproj} with
\begin{align*}
    \HdivsetZero{\widetilde\Patch} := \set{\ttau\in \Hdivset{\widetilde\Patch}}{\ttau\cdot\normal = 0 \text{ on }\partial\widetilde\Patch}.
\end{align*}
It remains to show that
\begin{align*}
  \sum_{\ell=0}^L \sum_{E\in\widetilde\EE_\ell^\Omega} \norm{\phi_{\ell,E}}{-s,\sim}^2 \lesssim \norm{\phi_*}{-s,\sim}^2. 
\end{align*}
Again, following the arguments in Section~\ref{sec:proof:main:lower:adaptive} we obtain the estimate
\begin{align}\label{eq:main:tilde:adap:2}
  \sum_{\ell=1}^L \norm{h_\ell^s (\projLtwo_\ell-\projLtwo_{\ell-1})\overline\projHone_\ell'\phi_*}{}^2 \lesssim  \sum_{m=0}^\infty \widehat h_m^{-2+2s} \norm{(1-\projTildeHmOneHat_\ell)\phi_*}{-1,\sim}^2. 
\end{align}
At this point it is important to note that to get the latter bound we use the following result instead of Lemma~\ref{lem:sumHmOne} combined with the local boundedness of $\overline\projHone_\ell'$ with respect to $\norm\cdot{-1,T,\Gamma}$ (see Lemma~\ref{lem:projTildeHoneDual}).
\begin{lemma}
  Let $\omega_1,\dots,\omega_N \subsetneq \Omega$ denote pairwise disjoint simply connected Lipschitz domains with positive measure and $|\partial\omega_j\setminus\Gamma|>0$.  Then,
  \begin{align*}
    \sum_{j=1}^N \norm{\psi|_{\omega_j}}{-1,\omega_j,\Gamma}^2 \lesssim \norm{\psi}{-1,\sim}^2
    \quad\text{for all } \psi\in L^2(\Omega),
  \end{align*}
  where the involved constant only depends on $\Omega$. 
\end{lemma}
\begin{proof}
  Let $v_j\in H_{\Gamma,c}^1(\omega_j)$ such that $\norm{\psi|_{\omega_j}}{-1,\omega_j,\Gamma}^2 = \ip{\psi}{v_j}_{\omega_j} = \norm{\nabla v_j}{\omega_j}^2$. Extend $v_j$ on $\Omega\setminus\omega_j$ by $0$ and note that $v_j\in H^1(\Omega)$ and $\norm{v_j}{H^1(\omega_j)}^2 = \norm{\nabla v_j}{\omega_j}^2 + \norm{v_j}{\omega_j}^2 \leq  \norm{\nabla v_j}{}^2 + \diam(\omega_j)^2 \norm{\nabla v_j}{}^2 \leq \norm{\nabla v_j}{}^2 + \diam(\Omega)^2 \norm{\nabla v_j}{}^2 \simeq \norm{\nabla v_j}{}^2$.
  Set $v=\sum_{j=1}^N v_j$ and observe that
\begin{align*}
  \sum_{j=1}^N \norm{\psi|_{\omega_j}}{-1,\omega_j,\Gamma}^2 = \ip{\psi}{v} \leq \norm{\psi}{-1,\sim}\norm{v}{H^1(\Omega)}.
\end{align*}
We have that $\norm{v}{H^1(\Omega)}^2 = \sum_{j=1}^N \norm{v_j}{H^1(\omega_j)}^2 \lesssim \sum_{j=1}^N \norm{\nabla v_j}{\omega_j}^2 = \sum_{j=1}^N \norm{\psi|_{\omega_j}}{-1,\omega_j,\Gamma}^2$ which concludes the proof.
\end{proof}

We continue to estimate~\eqref{eq:main:tilde:adap:2} by applying Theorem~\ref{thm:normequivInterpolation} which leads us to
\begin{align*}
  \sum_{\ell=1}^L \norm{h_\ell^s (\projLtwo_\ell-\projLtwo_{\ell-1})\overline\projHone_\ell'\phi_*}{}^2 \lesssim 
  \sum_{m=0}^\infty \widehat h_m^{-2+2s} \norm{(1-\projTildeHmOneHat_m)\phi_*}{-1,\sim}^2
  \lesssim \norm{\phi_*}{-1,\sim}^2.
\end{align*}
Similar arguments are used to estimate the last term in~\eqref{eq:main:tilde:adap} which concludes the proof.
\qed

\subsection{Proof of Theorem~\ref{thm:multilevelnorm} and Theorem~\ref{thm:multilevelnorm:tilde}}\label{sec:proof:main:norm}
We only give the sketch of the proof of Theorem~\ref{thm:multilevelnorm} as most arguments have already been given in Section~\ref{sec:proof:main:lower:adaptive} and Section~\ref{sec:proof:main:upper:adaptive}.
Moreover, the proof of Theorem~\ref{thm:multilevelnorm:tilde} follows a similar argumentation and is thus omitted. 
Note that in Section~\ref{sec:proof:main:lower:adaptive} we have shown that
\begin{align*}
  \sum_{\ell=1}^L \norm{h_\ell^s(\projHone_\ell'-\projHone_{\ell-1}')\phi}{}^2 \lesssim \norm{\phi}{-s}^2.
\end{align*}
Also note that $\norm{h_0^s\projHone_0'\phi}{} \lesssim \norm{\projHone_0'\phi}{-s} \lesssim \norm{\phi}{-s}$ which proves 
$\sum_{\ell=0}^L \norm{h_\ell^s(\projHone_\ell'-\projHone_{\ell-1}')\phi}{}^2 \lesssim \norm{\phi}{-s}^2$.

To see the other direction we consider
\begin{align}\label{eq:multilevelnorm:splitting}
  \phi = \projHone_L' \phi =  \sum_{\ell=0}^L (\projHone_\ell'-\projHone_{\ell-1}')\phi = 
  \projHone_0'\phi + \sum_{\ell=1}^L (\projHone_\ell'-\projHone_{\ell-1}')\phi =: \phi_0+\phi_1.
\end{align}
Clearly, $\norm{\phi}{-s}^2 \lesssim \norm{\phi_0}{-s}^2 + \norm{\phi_1}{-s}^2$ and $\norm{\phi_0}{-s} \lesssim \norm{\phi_0}{} \simeq \norm{h_0^{s}\phi_0}{}$.
As in Section~\ref{sec:proof:main:lower:adaptive} and Section~\ref{sec:proof:main:upper:adaptive} we use the sequence of uniform meshes $(\widehat\TT_m)_{m\in\N_0}$ with $\widehat\TT_0 =\TT_0$ and note that there exists a function $m'$ with $m'(\ell-1,z) = m$ such that $\diam(\Patch_{\ell-1}(z)) \simeq \widehat h_m$ and $\eta_{\ell-1,z} (\projHone_\ell'-\projHone_{\ell-1}')\phi \in \PP^2(\widehat\TT_m)$.
With the partition of unity, $1=\sum_{z\in\NN_{\ell-1}} \eta_{\ell-1,z}$, we consider
\begin{align*}
  \phi_1 = \sum_{\ell=1}^L (\projHone_\ell'-\projHone_{\ell-1}')\phi &= \sum_{\ell=1}^L \sum_{z\in\NN_{\ell-1}} \eta_{\ell-1,z}(\projHone_\ell'-\projHone_{\ell-1}')\phi \\
  &= \sum_{m=0}^M \sum_{\ell=1}^L \sum_{\substack{z\in\NN_{\ell-1}\\ m'(\ell-1,z)=m}} \eta_{\ell-1,z}(\projHone_\ell'-\projHone_{\ell-1}')\phi
  =: \sum_{m=0}^M \widehat\phi_m
\end{align*}
and note that $\widehat\phi_m\in\PP^2(\widehat\TT_m)$. Thus, we can apply Lemma~\ref{lem:decompositionUniform} which yields that
\begin{align}\label{eq:multilevelnorm:est2}
  \norm{\phi_1}{-s}^2 \lesssim \sum_{m=0}^M \widehat h_m^{-2+2s} \norm{\widehat \phi_m}{-1}^2.
\end{align}
We use the same observations as in Step~3 of the proof of Lemma~\ref{lem:stability:splittingPhiEll}:
For an interior node $z\in\NN_{\ell-1}^\Omega$ we have $\ip{\eta_{\ell-1,z}(\projHone_\ell'-\projHone_{\ell-1}')\phi}{1}=0$ and if $z\in\NN_{\ell-1}^\Gamma$ then $v \in H_0^1(\Omega)$ implies that $\norm{v}{\Patch_{\ell-1}(z)}\lesssim \diam(\Patch_{\ell-1}(z)) \norm{\nabla v}{\Patch_{\ell-1}(z)}$ since at least one facet of an element from $\patch_{\ell-1}(z)$ is a boundary facet. 

Let $v\in H_0^1(\Omega)$. Using the latter observations and similar arguments as in Sections~\ref{sec:proof:main:lower:adaptive}--\ref{sec:proof:main:tilde:adaptive}
we conclude that
\begin{align*}
  \ip{\widehat \phi_m}{v} &= \sum_{\ell=1}^L \sum_{\substack{z\in\NN_{\ell-1}\\ m'(\ell-1,z)=m}} 
  \ip{\eta_{\ell-1,z}(\projHone_\ell'-\projHone_{\ell-1}')\phi}{v}_{\Patch_{\ell-1}(z)} \\
  &\lesssim  \sum_{\ell=1}^L \sum_{\substack{z\in\NN_{\ell-1}, \eta_{\ell-1,z}(\projHone_\ell'-\projHone_{\ell-1}')\phi\neq 0\\ m'(\ell-1,z)=m}} \diam(\Patch_{\ell-1}(z)) 
  \norm{(\projHone_\ell'-\projHone_{\ell-1}')\phi}{\Patch_{\ell-1}(z)} \norm{\nabla v}{\Patch_{\ell-1}(z)} \\
  & \lesssim \Big(\sum_{\ell=1}^L \sum_{\substack{z\in\NN_{\ell-1}\\ m'(\ell-1,z)=m}} \widehat h_m^2
  \norm{(\projHone_\ell'-\projHone_{\ell-1}')\phi}{\Patch_{\ell-1}(z)}^2\Big)^{1/2}\norm{\nabla v}{}.
\end{align*}
Dividing by $\norm{\nabla v}{}$, taking the supremum over $0\neq v\in H_0^1(\Omega)$ together with~\eqref{eq:multilevelnorm:est2} gives
\begin{align*}
  \norm{\phi_1}{-s}^2 \lesssim \sum_{m=0}^M \sum_{\ell=1}^L \sum_{\substack{z\in\NN_{\ell-1}\\ m'(\ell-1,z)=m}} \widehat h_m^{2s} 
  \norm{(\projHone_\ell'-\projHone_{\ell-1}')\phi}{\Patch_{\ell-1}(z)}^2 \lesssim \sum_{\ell=1}^L \norm{h_\ell^s(\projHone_\ell'-\projHone_{\ell-1}')\phi}{}^2
\end{align*}
which finishes the proof.
\qed

\section{Numerical experiments}\label{sec:experiments}
In this section we present some numerical examples to support the results from Theorem~\ref{thm:main} and~\ref{thm:multilevelnorm}. 
All experiments have been realized using \textsc{Matlab} version 2017b on a Linux machine with an Intel i5-2520M processor and $8$ GB RAM. 

Let $\mathrm{conv}(\cdot,\cdot,\cdot)$ denote the convex hull. We consider the initial mesh
\begin{align*}
  \TT_0:= \left\lbrace \mathrm{conv}\big( z_1,z_2,z_5 \big), \,\mathrm{conv}\big( z_2,z_3,z_5 \big), \,\mathrm{conv}\big( z_3,z_4,z_5 \big), \,\mathrm{conv}\big( z_4,z_1,z_5 \big)  \right\rbrace,
\end{align*}
where $z_1,\dots,z_4$ denote the corners of the domain $\Omega = (0,1)^2$ and $z_5 = (\tfrac12,\tfrac12)$.
We consider two sequences of meshes, namely, uniform meshes $(\TT_\ell)_{\ell=0}^L=(\TT_\ell^\mathrm{unif})_{\ell=0}^L$ and locally refined meshes $(\TT_\ell)_{\ell=0}^L=(\TT_\ell^\mathrm{adap})_{\ell=0}^L$. 
The uniform mesh $\TT_{\ell+1}^\mathrm{unif}$ is created by dividing each element of $\TT_\ell^{\mathrm{unif}}$ into four son elements with equal area by applying (iteratively) the newest-vertex bisection rule. 
The adaptive mesh $\TT_{\ell+1}^\mathrm{adap}$ is created by marking all elements in the mesh $\TT_\ell^\mathrm{adap}$ which share the vertex $z_5$ for refinement and then applying the routine given in~\cite[Listing 5.2]{P1AFEM} (which also employs newest-vertex bisection).

In Section~\ref{sec:experiments:precond} we present results for the preconditioner induced by the decomposition from Theorem~\ref{thm:main} and in Section~\ref{sec:experiments:norm} we show results corresponding to Theorem~\ref{thm:multilevelnorm}.
In both cases we use a discrete $H^{-s}(\Omega)$ norm which is defined by following the ideas from~\cite{ArioliLoghin09}:
First, let $\Hmat_0$ denote the $L^2(\Omega)$ Riesz matrix with respect to the canonical basis $\chi_{T_1},\dots,\chi_{T_{\#\TT_L}}$ of $\PP^0(\TT_L)$.
Second, let $\Hmat_{-1}$ denote the Riesz matrix of the discrete $H^{-1}(\Omega)$ inner product from~\cite[Section~7.1]{FuehrerHeuerQuasiDiagonal19}, given by
\begin{align*}
  \Hmat_{-1} = \Mmat^\top \Hmat_1^{-1} \Mmat + \beta \widetilde\Hmat_0.
\end{align*}
Here, the entries of the above matrices are given by
\begin{align*}
  \Mmat[j,\ell] := \ip{\chi_{T_\ell}}{\eta_{L,z_j}}, \quad 
  \Hmat_1[j,k] := \ip{\nabla \eta_{L,z_j}}{\nabla \eta_{L,z_k}}, \quad
  \widetilde\Hmat_0[m,\ell]:=\ip{\widetilde h_L^2 \chi_{T_\ell}}{\chi_{T_m}}
\end{align*}
for $j,k=1,\dots,\dim(\cS_0^1(\TT_L))$, $\ell,m=1,\dots,\#\TT_L$ and $\eta_{L,z_j}$ denotes the nodal basis of $\cS_0^1(\TT_L)$. 
Moreover, $\widetilde h_L|_T := |T|^{1/2}$ is an equivalent mesh-size function and $\beta\simeq 1$ can be freely chosen. For the following experiments we choose $\beta=\tfrac15$.
Then, by~\cite[Proposition~3.2]{ArioliLoghin09} the matrix
\begin{align*}
  \Hmat_{-s} := \Qmat^\top \Rmat^{1-s} \Qmat
\end{align*}
defines a discrete $H^{-s}(\Omega)$ inner product which is equivalent to $\norm{\cdot}{-s}$, i.e., 
\begin{align*}
  \norm{\phi}{-s,h}^2:=\xx^\top\Hmat_{-s} \xx \simeq \norm{\phi}{-s}^2  \quad\text{for all } \phi \in\PP^0(\TT_L) \text{ with coefficient vector } \xx.
\end{align*}
Here, $\Rmat$ is a diagonal matrix with positive entries and $\Qmat$ is invertible with $\Hmat_{-1}^{-1}\Hmat_0 = \Qmat^{-1}\Rmat\Qmat$, see~\cite[Section~3.1]{ArioliLoghin09} for a detailed description.
We note that the latter equivalence result requires the existence of a projection operator onto $\PP^0(\TT_L)$ which is bounded in $H^{-1}(\Omega)$ and $L^2(\Omega)$, see~\cite[Lemma~2.3]{ArioliLoghin09}. In our situation such an operator is given by $\projHmOneSZ_L$ (Theorem~\ref{thm:projHmOneSZ}).

\subsection{Multilevel preconditioner in $H^{-s}(\Omega)$}\label{sec:experiments:precond}
\begin{figure}
  \begin{center}
    \includegraphics[width=0.42\textwidth]{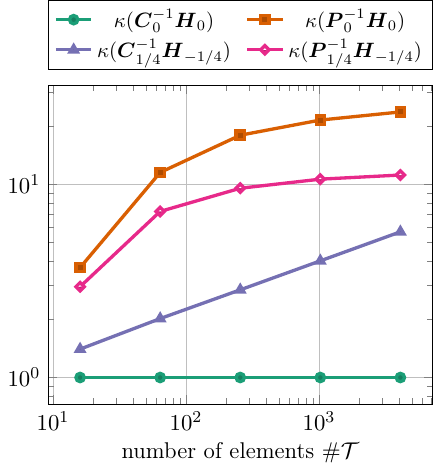}
    \includegraphics[width=0.42\textwidth]{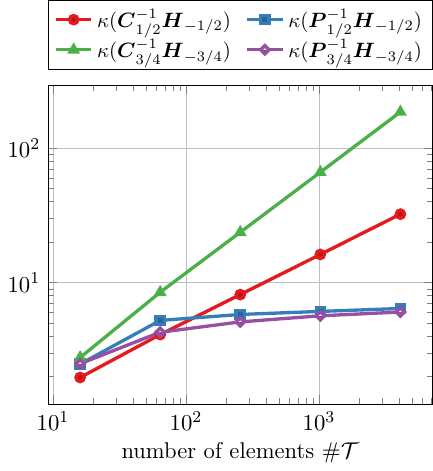}
  \end{center}
  \caption{Condition numbers in the case of uniform meshes for $s=0,\tfrac14$ (left) and $s=\tfrac12,\tfrac34$ (right).}
  \label{fig:unif}
\end{figure}
\begin{figure}
  \begin{center}
    \includegraphics[width=0.42\textwidth]{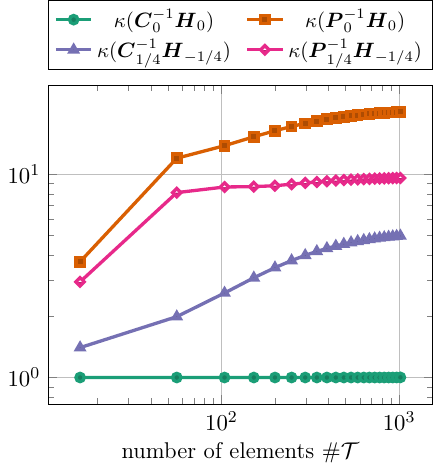}
    \includegraphics[width=0.42\textwidth]{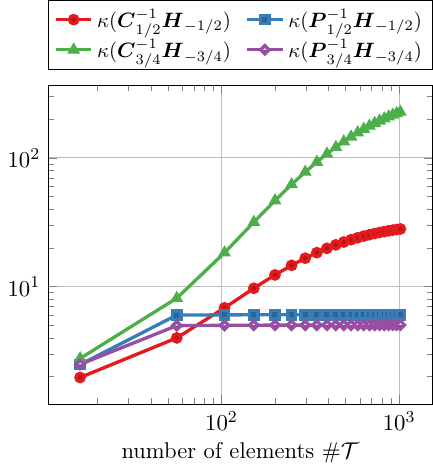}
  \end{center}
  \caption{Condition numbers in the case of adaptive meshes for $s=0,\tfrac14$ (left) and $s=\tfrac12,\tfrac34$ (right).}
  \label{fig:adap}
\end{figure}
We start with a description of the matrix representation of the preconditioner associated to the multilevel decomposition from Theorem~\ref{thm:main}. 
We use similar notations and definitions as in~\cite[Section~3.1]{ABEMsolve}.
Let $\Imat_\ell\in \R^{\#\TT_L\times \#\TT_\ell}$ denote the matrix representation of the embedding $\PP^0(\TT_\ell)\hookrightarrow \PP^0(\TT_L)$ and let $\Tmat_\ell \in\R^{\#\TT_\ell\times\#\EE_\ell}$ denote the representation of the Haar-type functions $\psi_{\ell,E}$, i.e., 
\begin{align*}
  \psi_{\ell,E_j} = \sum_{i=1}^{\#\TT_\ell}\Tmat_\ell[i,j]\chi_{\ell,T_i} \quad\text{for all } E_j\in\EE_\ell.
\end{align*}
Note that $\Tmat_\ell$ is sparse since $\psi_{\ell,E_j}$ is supported on at most two elements. Furthermore, let $\Dmat_{\ell,s} \in \R^{\#\EE_\ell\times \#\EE_\ell}$ denote the diagonal matrix with entries
\begin{align*}
  \Dmat_{\ell,s}[j,k] = \begin{cases}
    \norm{\psi_{\ell,E_j}}{-s,h}^{-2} & \text{if } j=k \text{ and } E_j\in \widetilde\EE_\ell, \\
    0 & \text{else}.
  \end{cases}
\end{align*}
Recall that $\norm{\cdot}{-s,h}$ is the discrete $H^{-s}(\Omega)$ norm induced by the matrix $\Hmat_{-s}$. 
The matrix representation of the multilevel preconditioner then reads
\begin{align*}
  \Pmat_s^{-1} = \sum_{\ell=0}^L \Imat_\ell\Tmat_\ell \Dmat_{\ell,s} \Tmat_\ell^\top \Imat_\ell^\top. 
\end{align*}
It follows from Theorem~\ref{thm:main} and the additive Schwarz theory, see Section~\ref{sec:abstract}, that the spectral condition number of the preconditioned matrix $\Pmat_s^{-1}\Hmat_{-s}$ is uniformly bounded, i.e, 
\begin{align*}
  \kappa(\Pmat_s^{-1}\Hmat_{-s})\lesssim 1.
\end{align*}
In the experiments we also consider the diagonal preconditioner $\Cmat_s^{-1} := \mathrm{diag}(\Hmat_{-s})^{-1}$.
Figure~\ref{fig:unif} and Figure~\ref{fig:adap} show the results for different values of $s$ for uniform and adaptive meshes, respectively. 
Note that the case $s=0$ is not covered in Theorem~\ref{thm:main}, cf. Remark~\ref{rem:extension}.
Although the condition numbers of $\Pmat_s^{-1}\Hmat_{-s}$ for $s=0$, $s=1/4$ are higher than in the case of the diagonally preconditioned matrices it seems that they reach an asymptotic uniform bound supporting the result from Theorem~\ref{thm:main}. 
We stress that in the case of the diagonal preconditioner the condition numbers are of order $\OO( (\#\TT_L)^s)$, see~\cite{amt99} for details on diagonally preconditioned systems for problems in fractional order Sobolev spaces. 
For $s=\tfrac12$ and $s=\tfrac34$ we observe that our proposed preconditioner outperforms the simple diagonal scaling and yields uniformly bounded condition numbers on uniform as well as locally refined meshes.

\subsection{Multilevel norm in $H^{-s}(\Omega)$}\label{sec:experiments:norm}
\begin{figure}
  \begin{center}
    \includegraphics[width=0.42\textwidth]{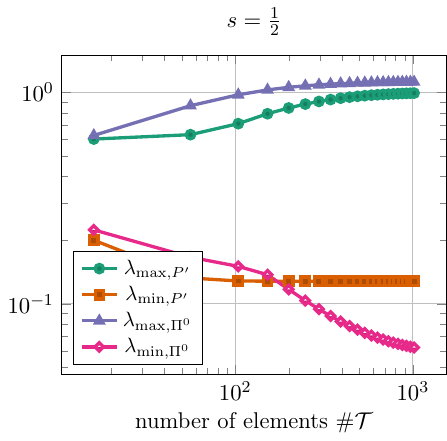}
    \includegraphics[width=0.42\textwidth]{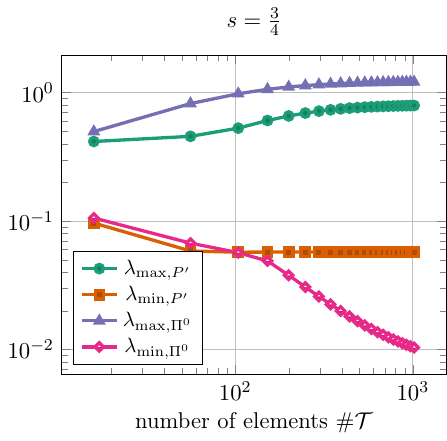}
  \end{center}
  \caption{Squared norm equivalence constants in the case of adaptive meshes for $s=\tfrac12,\tfrac34$ for the multilevel norms from Section~\ref{sec:experiments:norm}.}
  \label{fig:mln}
\end{figure}
Define the matrices $\Bmat_{s,\projLtwo}, \Bmat_{s,\projHone'}\in\R^{\#\TT_L\times\#\TT_L}$ by
\begin{align*}
  \Bmat_{s,\projLtwo}[j,k] &:= \sum_{\ell=0}^L \ip{\widetilde h_\ell^{2s}(\projLtwo_\ell-\projLtwo_{\ell-1})\chi_{T_k}}{(\projLtwo_\ell-\projLtwo_{\ell-1})\chi_{T_j}}, \\
  \Bmat_{s,\projHone'}[j,k] &:= \sum_{\ell=0}^L \ip{\widetilde h_\ell^{2s}(\projHone'_\ell-\projHone'_{\ell-1})\chi_{T_k}}{(\projHone'_\ell-\projHone'_{\ell-1})\chi_{T_j}}
\end{align*}
for all $j,k=1,\dots,\#\TT_L$. 
Note that $\Bmat_{s,\projLtwo}$ corresponds to the multilevel norm from~\eqref{eq:mlnL2} and $\Bmat_{s,\projHone'}$ to the one from Theorem~\ref{thm:multilevelnorm}. (The only difference is the use of the equivalent mesh-size function $\widetilde h_\ell$.)
The main idea of the numerical experiments in this section is to compute the optimal constants $\lambda_{\mathrm{min},\star}$, $\lambda_{\mathrm{max},\star}$ for $\star\in\{\projLtwo,\projHone'\}$ that satisfy
\begin{align*}
  \lambda_{\mathrm{min},\star} \xx^\top\Bmat_{s,\star}\xx \leq \xx^\top \Hmat_{-s}\xx \leq \lambda_{\mathrm{max},\star}\xx^\top\Bmat_{s,\star}\xx \quad\text{for all }\xx\in \R^{\#\TT_L}.
\end{align*}
Figure~\ref{fig:mln} shows the results for adaptive meshes and $s=\tfrac12,\tfrac34$. 
Note that we expect that $\lambda_{\mathrm{min},\projLtwo}$ deteriorates, see~\cite[Theorem~2]{Oswald98} and Section~\ref{sec:knownresults:multilevel} for the case of uniform meshes and $s\geq \tfrac12$.
This can be seen in our results also for the adaptive meshes under consideration.
Contrary, Theorem~\ref{thm:multilevelnorm} predicts uniformly bounded $\lambda_{\mathrm{min},\projHone'}$, $\lambda_{\mathrm{max},\projHone'}$ which is also observed in Figure~\ref{fig:mln}.

\section{Concluding remarks}\label{sec:conclusion}
First, all given theorems in Section~\ref{sec:main} are valid if we replace the Lipschitz domain $\Omega$ with a regular manifold $\widetilde\Gamma \subseteq\Gamma = \partial\Omega$.
The proofs are almost identical with some minor modifications. The most notable are the use of Raviart--Thomas surface elements, the definition of the corresponding operator from Lemma~\ref{lem:Hdivproj}, and the use of local Laplace--Beltrami problems in~\eqref{eq:localProblemPhi1} resp.~\eqref{eq:localProblemPhi2}.
In our work~\cite{ABEMsolve} we considered the case of a closed manifold $\widetilde\Gamma$ and proved Theorem~\ref{thm:main:tilde} for $s=1/2$ using different techniques (we constructed extension operators into spaces associated to the volume $\Omega$ similar as in~\cite{HiptmairJerezMao15}). 
Numerical examples in~\cite[Section~4]{ABEMsolve} provide the numerical evidence of the optimality of the preconditioners associated to the multilevel decompositions for the case $s=1/2$. We stress that in~\cite{ABEMsolve} we did not prove optimality for open manifolds but only claimed it~\cite[Remark~4]{ABEMsolve}. Thus, Theorem~\ref{thm:main:tilde} provides the mathematical proof of this claim which is supported by numerical experiments~\cite[Section~4.4]{ABEMsolve}.

Second, concerning implementation of the preconditioners corresponding to Theorem~\ref{thm:main} resp. Theorem~\ref{thm:main:tilde} it is well-known that multilevel decompositions based on one-dimensional subspaces lead to (local) multilevel diagonal scaling preconditioners which are utmost simple to implement. 
Moreover, the preconditioners can be evaluated in $\OO(\#\TT_L)$ operations and the storage requirements are $\OO(\#\TT_L)$ units. 
We refer to~\cite[Section~3]{ABEMsolve} for a short discussion. 

Third, concerning implementation of the multilevel norms from Theorem~\ref{thm:multilevelnorm} resp. Theorem~\ref{thm:multilevelnorm:tilde} we note that the local definition of the involved operators imply that $(\projHone_\ell'-\projHone_{\ell-1}')\phi$ is supported only in a neighborhood of $\TT_\ell\setminus\TT_{\ell-1}$ and therefore the multilevel norms can be evaluated in $\OO(\#\TT_L)$ operations. Contrary to~\cite{ArioliLoghin09} our proposed multilevel norms do not rely on the evaluation of powers of a matrix.

Fourth, the decompositions $\XX_L$ resp.  $\widetilde\XX_L$ in Theorem~\ref{thm:main} resp. Theorem~\ref{thm:main:tilde} can be replaced by
\begin{align*}
  \XX_L &:= \{\PP^0(\TT_0)\} \cup \set{\XX_{\ell,E}}{E\in\widetilde\EE_\ell,\,\ell=1,\dots,L} \quad\text{resp.}\\
  \widetilde\XX_L &:= \{\PP^0(\TT_0)\} \cup \set{\XX_{\ell,E}}{E\in\widetilde\EE_\ell^\Omega,\,\ell=1,\dots,L}.
\end{align*}
Note that the additional space $\PP^0(\TT_0)$ necessitates the inversion of the Riesz matrix corresponding to the inner product $\ip\cdot\cdot_{-s}$ resp. $\ip\cdot\cdot_{-s,\sim}$ or a discrete one on the coarsest level when implementing the preconditioners. However, tighter equivalence constants are expected.

Fifth, decompositions resp. multilevel norms for polynomial discretization spaces of higher order can be handled using the following observations: 
For some fixed $p\geq 1$ consider
\begin{align*}
  \PP^p(\TT_L) = \PP^0(\TT_L) \oplus \PP_*^p(\TT_L)
\end{align*}
where $\PP_*^p(\TT_L)$ is $L^2(\Omega)$ orthogonal to $\PP^0(\TT_L)$. Let $\{\chi_{\widehat T,1},\dots,\chi_{\widehat T,d_p-1}\}$ denote a basis of $\PP_*^p(\widehat T)$ with $\widehat T$ being a reference element and $d_p =\dim(\PP^p(T))$. Using affine transformations this basis defines a basis $\{\chi_{T,1},\dots,\chi_{T,d_p-1}\}$ of $\PP_*^p(T)$ for all $T\in\TT_L$. 
Then, 
\begin{align*}
  \norm{\phi}{-s}^2 \simeq \norm{\phi_0}{-s}^2 + \sum_{T\in\TT_L} \sum_{j=1}^{d_p-1} \norm{\phi_{T,j}}{-s}^2 .
\end{align*}
for $\phi = \phi_0 + \sum_{T\in\TT_L} \sum_{j=1}^{d_p-1} \phi_{T,j} \in\PP^p(\TT_L)$ with $\phi_0\in\PP^0(\TT_L)$ and $\phi_{T,j}\in\linhull\{\chi_{T,j}\}$. 
The case $s=0$ is trivial, the case $s=1$ can be seen from boundedness of $\projLtwo_L$ restricted to polynomials and inverse estimates see, e.g.~\cite[Lemma~9]{FuehrerHeuerQuasiDiagonal19}. The general case $s\in(0,1)$ is derived from the latter two. 
The latter equivalence holds replacing $\norm\cdot{-s}$ with $\norm\cdot{-s,\sim}$, thus, we conclude:
\begin{corollary}
  Let $p\in\N_0$. 
  Theorem~\ref{thm:main} and Theorem~\ref{thm:main:tilde} remain valid if $\PP^0(\TT_L)$ is replaced with $\PP^p(\TT_L)$ and the decompositions $\XX_L$ resp. $\widetilde\XX_L$ are replaced with
  \begin{align*}
    \XX_L^p &:=  \XX_L \cup \set{\linhull\{\chi_{T,j}\}}{j=1,\dots,d_p-1, \, T\in\TT_L}
    \quad\text{resp.}
    \\
    \widetilde\XX_L^p &:= \widetilde\XX_L \cup \set{\linhull\{\chi_{T,j}\}}{j=1,\dots,d_p-1, \, T\in\TT_L}.
  \end{align*}
  The involved constants additionally depend on $p\in\N_0$ and the basis of $\PP_*^p(\widehat T)$.
\end{corollary}
Similarly, we adapt the multilevel norms from Theorem~\ref{thm:multilevelnorm} resp.~\ref{thm:multilevelnorm:tilde}:
\begin{corollary}
  Let $p\in\N_0$, $s\in(0,1)$. Then,
  \begin{align*}
    \norm{\phi}{-s}^2 &\simeq \sum_{\ell=0}^L \norm{h_\ell^s(\projHone_\ell'-\projHone_{\ell-1}')\phi_0}{}^2 
    + \sum_{T\in\TT_L} \sum_{j=1}^{d_p-1} h_T^{2s} \norm{\phi_{T,j}}{T}^2, \\
    \norm{\phi}{-s,\sim}^2 &\simeq \sum_{\ell=0}^L \norm{h_\ell^s(\overline\projHone_\ell'-\overline\projHone_{\ell-1}')\phi_0}{}^2 
    + \sum_{T\in\TT_L} \sum_{j=1}^{d_p-1} h_T^{2s} \norm{\phi_{T,j}}{T}^2
  \end{align*}
  for all $\phi := \phi_0 + \phi_1 := \projLtwo_L\phi + (1-\projLtwo_L)\phi\in\PP^p(\TT_L)$ with $(1-\projLtwo_L)\phi = \sum_{T\in\TT_L}\sum_{j=1}^{d_p-1} \phi_{T,j}$.
  The involved constants depend only on $\Omega$, $s$, $d$, $p$, the constants from~\ref{ass:mesh:reg}--\ref{ass:mesh:gen}, $\TT_0$, and the basis of $\PP_*^p(\widehat T)$.
\end{corollary}

Finally, besides the already mentioned applications in preconditioning, the presented multilevel norms can be used in minimization problems involving negative order Sobolev spaces, which will be reported in future works.

\bibliographystyle{abbrv}
\bibliography{literature}

\end{document}

%% file: header.tex
\newtheorem{theorem}{Theorem}
\newtheorem{lemma}[theorem]{Lemma}
\newtheorem{corollary}[theorem]{Corollary}

\newtheorem{remark}[theorem]{Remark}

\newcommand{\patch}{\omega}
\newcommand{\Patch}{\Omega}

\newcommand{\II}{\mathcal{I}}
\newcommand{\HHH}{\mathcal{H}}

\newcommand{\Mmat}{\boldsymbol{M}}

\newcommand{\Bmat}{\boldsymbol{B}}
\newcommand{\Cmat}{\boldsymbol{C}}
\newcommand{\Dmat}{\boldsymbol{D}}
\newcommand{\Pmat}{\boldsymbol{P}}

\newcommand{\Hmat}{\boldsymbol{H}}
\newcommand{\Imat}{\boldsymbol{I}}
\newcommand{\Rmat}{\boldsymbol{R}}
\newcommand{\Tmat}{\boldsymbol{T}}
\newcommand{\Qmat}{\boldsymbol{Q}}

\newcommand{\XX}{\mathcal{X}}

\DeclareMathOperator{\gen}{gen}
\DeclareMathOperator{\linhull}{span}

\DeclareMathOperator{\supp}{supp}
\def\enorm#1{|\hspace*{-.5mm}|\hspace*{-.5mm}|#1|\hspace*{-.5mm}|\hspace*{-.5mm}|}

\newcommand{\interp}[3]{[#1,#2]_{#3}}
\newcommand{\projHmOne}{\Pi^{(-1)}}
\newcommand{\projHmOneHat}{\widehat\Pi^{(-1)}}
\newcommand{\projTildeHmOne}{\Pi^{(-1),\sim}}
\newcommand{\projTildeHmOneHat}{\widehat\Pi^{(-1),\sim}}

\newcommand{\projSZ}{J}
\newcommand{\projHone}{P}
\newcommand{\projHmOneSZ}{Q}
\newcommand{\projBubble}{B}
\newcommand{\projLtwo}{\Pi^0}
\newcommand{\projHdiv}{\boldsymbol{P}}

\newcommand{\ip}[2]{(#1\hspace*{.5mm},#2)}
\newcommand{\dual}[2]{\langle#1\hspace*{.5mm},#2\rangle}

\newcommand{\norm}[3][]{#1\|#2#1\|_{#3}}

\newcommand{\diam}{\mathrm{diam}}

\def\div{{\rm div\,}}

\def\intr{{\rm int}}

\newcommand{\Hdivset}[1]{\boldsymbol{H}(\div;#1)}
\newcommand{\HdivsetZero}[1]{\boldsymbol{H}_0(\div;#1)}

\newcommand{\set}[2]{\big\{#1\,:\,#2\big\}}

\newcommand{\RT}{\ensuremath{\mathcal{RT}}}
\newcommand{\mfun}{\ensuremath{\widetilde{m}}}
\newcommand{\mfunUp}{\ensuremath{\overline{m}}}
\newcommand{\tfun}{\ensuremath{\widetilde{t}}}

\newcommand{\R}{\ensuremath{\mathbb{R}}}
\newcommand{\N}{\ensuremath{\mathbb{N}}}

\newcommand{\TT}{\ensuremath{\mathcal{T}}}

\newcommand{\cS}{\ensuremath{\mathcal{S}}}

\newcommand{\PP}{\ensuremath{\mathcal{P}}}

\newcommand{\OO}{\ensuremath{\mathcal{O}}}
\newcommand{\EE}{\ensuremath{\mathcal{E}}}
\newcommand{\NN}{\ensuremath{\mathcal{N}}}

\newcommand{\normal}{\ensuremath{{\boldsymbol{n}}}}

\newcommand{\ppsi}{{\boldsymbol\psi}}

\newcommand{\ssigma}{{\boldsymbol\sigma}}
\newcommand{\ttau}{{\boldsymbol\tau}}

\newcommand{\xx}{\boldsymbol{x}}